\documentclass[11pt,reqno,tbtags]{amsart}
\pdfoutput=1
\usepackage[]{amsmath,amssymb,amsfonts,latexsym,amsthm,enumerate}
\usepackage{amssymb,bbm,mathrsfs}
\usepackage{enumerate,graphicx,paralist}
\usepackage[numeric,initials,nobysame]{amsrefs}

\numberwithin{equation}{section}
\usepackage[marginratio=1:1,height=584pt,width=488pt,tmargin=117pt]{geometry}
\usepackage[margin=25pt,font=small,justification=centering
]{caption}

\newtheorem{maintheorem}{Theorem}

\newtheorem{theorem}{Theorem}[section]
\newtheorem*{theorem*}{Theorem}
\newtheorem{lemma}[theorem]{Lemma}
\newtheorem{claim}[theorem]{Claim}
\newtheorem{proposition}[theorem]{Proposition}

\theoremstyle{definition}{

\newtheorem*{definition*}{Definition}

\newtheorem*{example*}{Example}
\newtheorem{remark}[theorem]{Remark}
\newtheorem*{remark*}{Remark}
}

\newcommand{\R}{\mathbb R}

\newcommand{\Z}{\mathbb Z}

\newcommand{\deq}{\stackrel{\scriptscriptstyle\triangle}{=}}

\newcommand{\E}{\mathbb{E}}
\renewcommand{\P}{\mathbb{P}}
\DeclareMathOperator{\var}{Var}
\DeclareMathOperator{\Cov}{Cov}

\newcommand{\tmix}{t_\textsc{mix}}

\newcommand{\tv}{{\textsc{tv}}}
\newcommand{\Po}{\operatorname{Po}}

\newcommand{\one}{\mathbbm{1}}
\newcommand{\red}{\textsc{Red}}
\newcommand{\blue}{\textsc{Blue}}
\newcommand{\green}{\textsc{Green}}

\renewcommand{\epsilon}{\varepsilon}
\renewcommand{\phi}{\varphi}

\newcommand{\cG}{\mathcal{G}}

\newcommand{\cC}{\mathcal{C}}

\newcommand{\cE}{\mathcal{E}}

\newcommand{\cJ}{\mathcal{J}}

\newcommand{\cF}{\mathcal{F}}
\newcommand{\cM}{\mathcal{M}}

\newcommand{\cU}{\mathcal{U}}

\newcommand{\cX}{\mathcal{X}}

\DeclareMathOperator{\sign}{sign}

\newcommand{\len}{{\mathfrak{L}}}
\newcommand{\sm}{{\mathfrak{m}}}
\newcommand{\anim}{{\mathfrak{W}}}

\newcommand{\uni}{{(\textsc{u})}}

\newcommand{\sH}{\mathscr{H}}
\newcommand{\hsH}{\hat{\sH}}
\newcommand{\hcC}{\hat{\cC}}
\newcommand{\htau}{\hat{\tau}}
\newcommand{\hcX}{\hat{\cX}}

\newcommand{\tcut}{t_\sm}
\newcommand{\scut}{s_\star}
\newcommand{\tpluss}{t_\star}
\newcommand{\tminuss}{t_\star^-}

\date{}

\begin{document}
\title{Universality of cutoff for the Ising model}

\author{Eyal Lubetzky}
\address{Eyal Lubetzky\hfill\break
Microsoft Research\\
One Microsoft Way\\
Redmond, WA 98052-6399, USA.}
\email{eyal@microsoft.com}
\urladdr{}

\author{Allan Sly}
\address{Allan Sly\hfill\break
Department of Statistics\\
UC Berkeley\\
Berkeley, CA 94720, USA.}
\email{sly@stat.berkeley.edu}
\urladdr{}

\begin{abstract}
 On any locally-finite geometry, the stochastic Ising model is known to be contractive when the inverse-temperature $\beta$ is small enough, via classical results of Dobrushin and of Holley in the 1970's. By a general principle proposed by Peres, the dynamics is then expected to exhibit cutoff. However, so far cutoff for the Ising model has been confirmed mainly for lattices, heavily relying on amenability and log Sobolev inequalities. Without these, cutoff was unknown at any fixed $\beta>0$, no matter how small, even in basic examples such as the Ising model on a binary tree or a random regular graph.

We use the new framework of information percolation to show that, in any geometry, there is cutoff for the Ising model at high enough temperatures. Precisely, on any sequence of graphs with maximum degree $d$, the Ising model has cutoff provided that $\beta<\kappa/d$ for some absolute constant $\kappa$ (a result which, up to the value of $\kappa$, is best possible). Moreover, the cutoff location is established as the time at which the sum of squared magnetizations drops to 1, and the cutoff window is $O(1)$, just as when $\beta=0$.

Finally, the mixing time from almost every initial state is not more than a factor of $1+\epsilon_\beta$ faster then the worst one (with $\epsilon_\beta\to0$ as $\beta\to 0$), whereas the uniform starting state is at least $2-\epsilon_\beta$ times faster.
\end{abstract}

\maketitle

\vspace{-0.75cm}

\section{Introduction}

Classical results going back to Dobrushin~\cite{Dobrushin} and to Holley~\cite{Holley1} in the early 1970's and continuing with the works of Dobrushin and Shlosman~\cite{DoSh} and of Aizenman and Holley~\cite{AH} show that, if $G$ is any graph on $n$ vertices with maximum degree $d$, the Glauber dynamics for the Ising model on $G$ exhibits a rapid convergence to equilibrium in total-variation distance at high enough temperatures. Namely, if the inverse-temperature $\beta$ is at most $c_0 /d$ for some absolute $c_0>0$ then the continuous-time dynamics is contractive, whence coupling techniques show that the total-variation mixing time is $O(\log n)$.

A known consequence of contraction is that the spectral gap of the dynamics is bounded away from 0, and so, by a general principle proposed by Peres in 2004 (addressing whether or not the product of the spectral gap and mixing time diverges with $n$), one expects the \emph{cutoff phenomenon}\footnote{sharp transition in the
$L^1$-distance of a finite Markov chain from equilibrium, dropping quickly
from near 1 to near 0.} to occur. (For more on the cutoff phenomenon, discovered in the early 80's by Aldous and Diaconis, see~\cites{AD,Diaconis}.)
Concretely, Peres conjectured~(\cite{LLP}*{Conjecture~1},\cite{LPW}*{\S23.2}) cutoff for the Ising model on any sequence of transitive graphs when the mixing time is $O(\log n)$, and in particular in the range $\beta < c_0/d$ as above.

This universality principle, whereby cutoff should accompany high enough temperatures in any underlying geometry, is supported by the heuristic that at small enough $\beta$ the model should qualitatively behave as if $\beta=0$. The latter, equivalent to random walk on the hypercube,
was one of the first examples of cutoff, established with an $O(1)$-cutoff window by Aldous~\cite{Aldous}, and refined in~\cites{DiSh2,DGM}. Thus, one may further expect cutoff for the Ising model with an $O(1)$-window provided that $\beta$ is small enough.

In contrast, cutoff for the Ising model has so far mainly been confirmed on $\Z^d$~\cites{LS1,LS3}, via proofs that hinged on log-Sobolev inequalities (see~\cites{DS1,DS2,DS,SaloffCoste}) that are known to hold for the Ising model on the lattice~\cites{HoSt1,MO,MO2,MOS,Martinelli97,SZ1,SZ3} as well as on the sub-exponential growth rate of balls in the lattice.

\begin{figure}[t]
\raisebox{-2.5cm}{
\includegraphics[width=.3\textwidth]{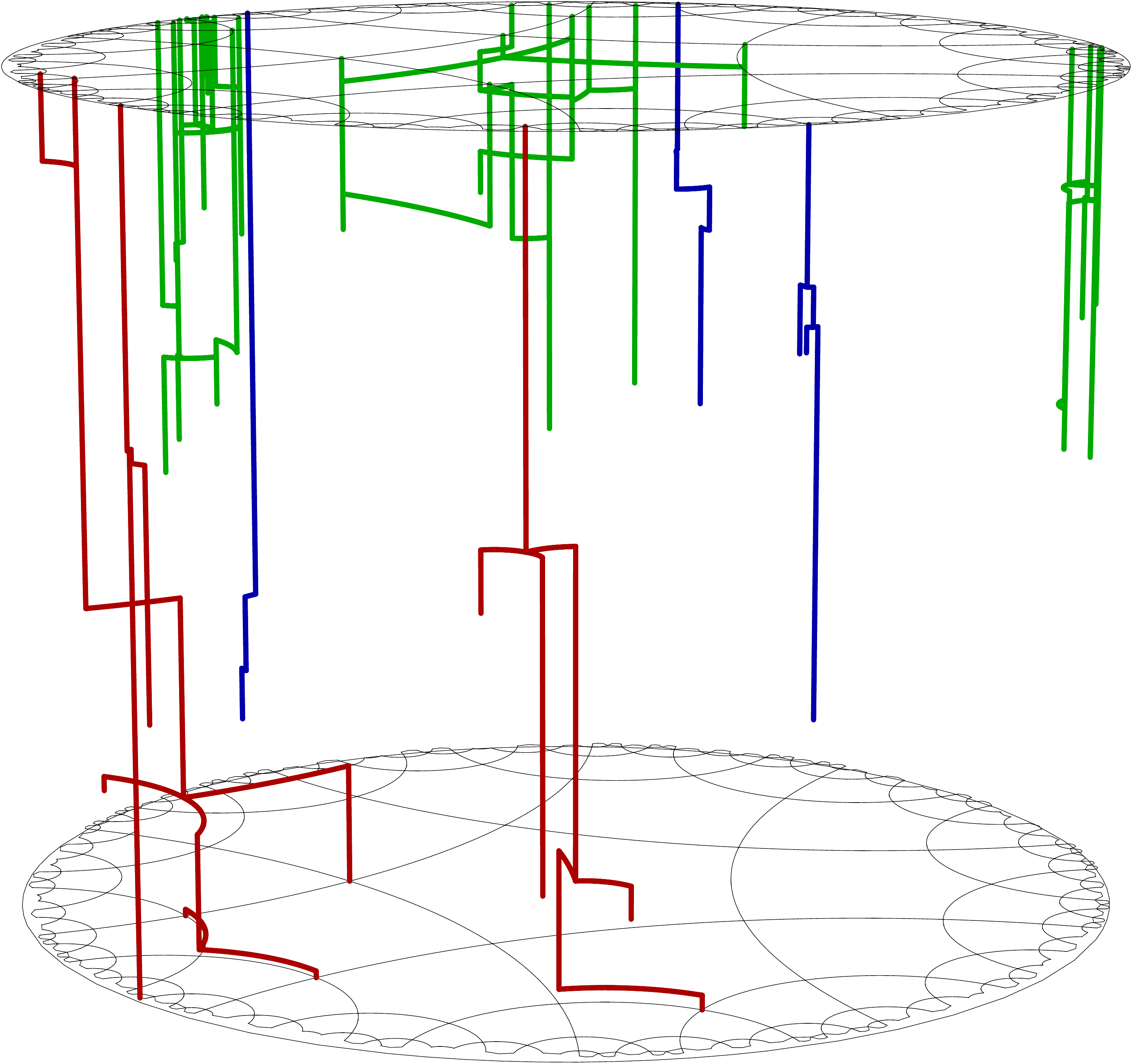}}
\hspace{.75cm}
\begin{tabular}{l}
\includegraphics[width=.55\textwidth]{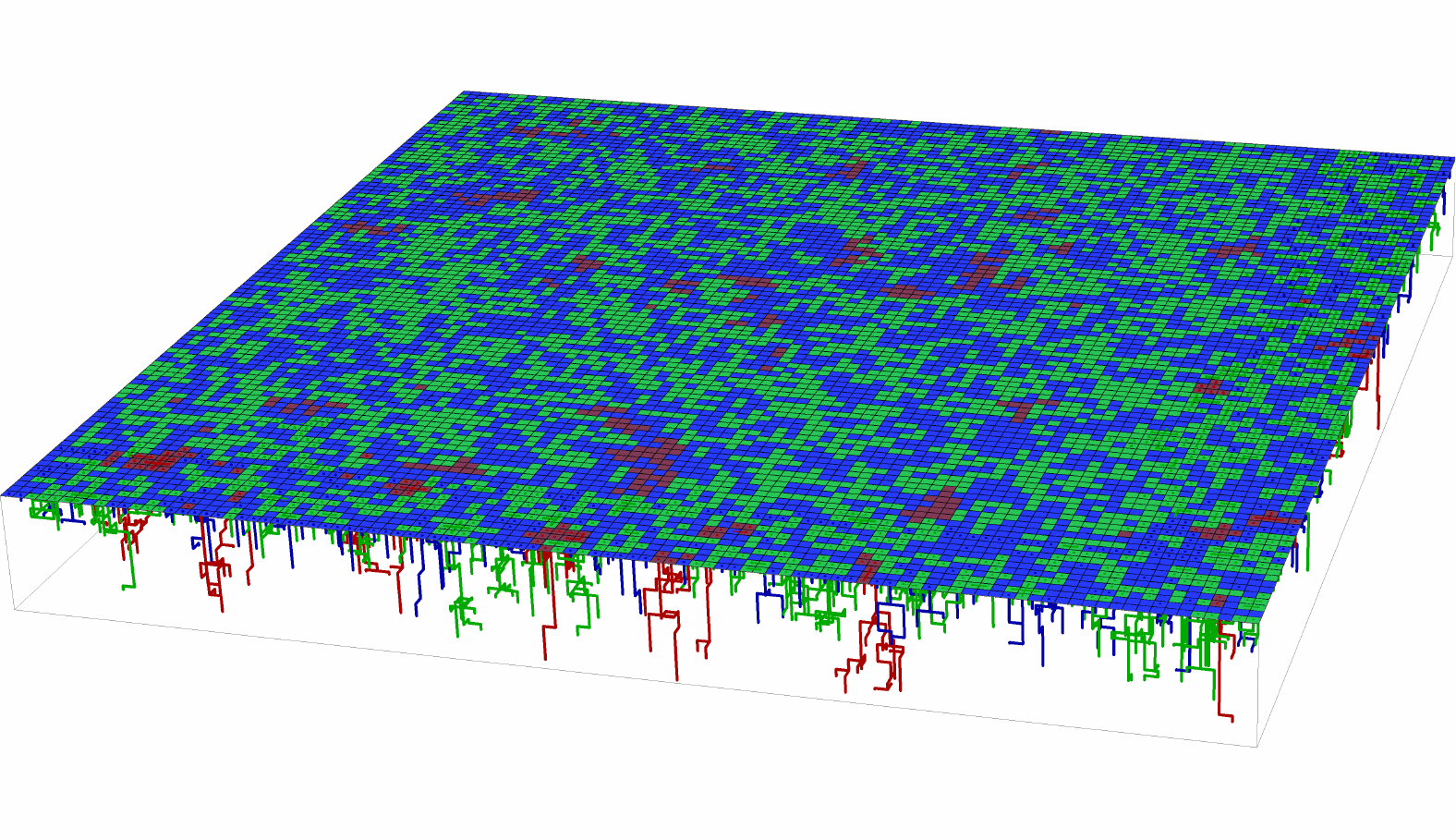} \vspace{-1.1cm}
\\
\includegraphics[width=.195\textwidth]{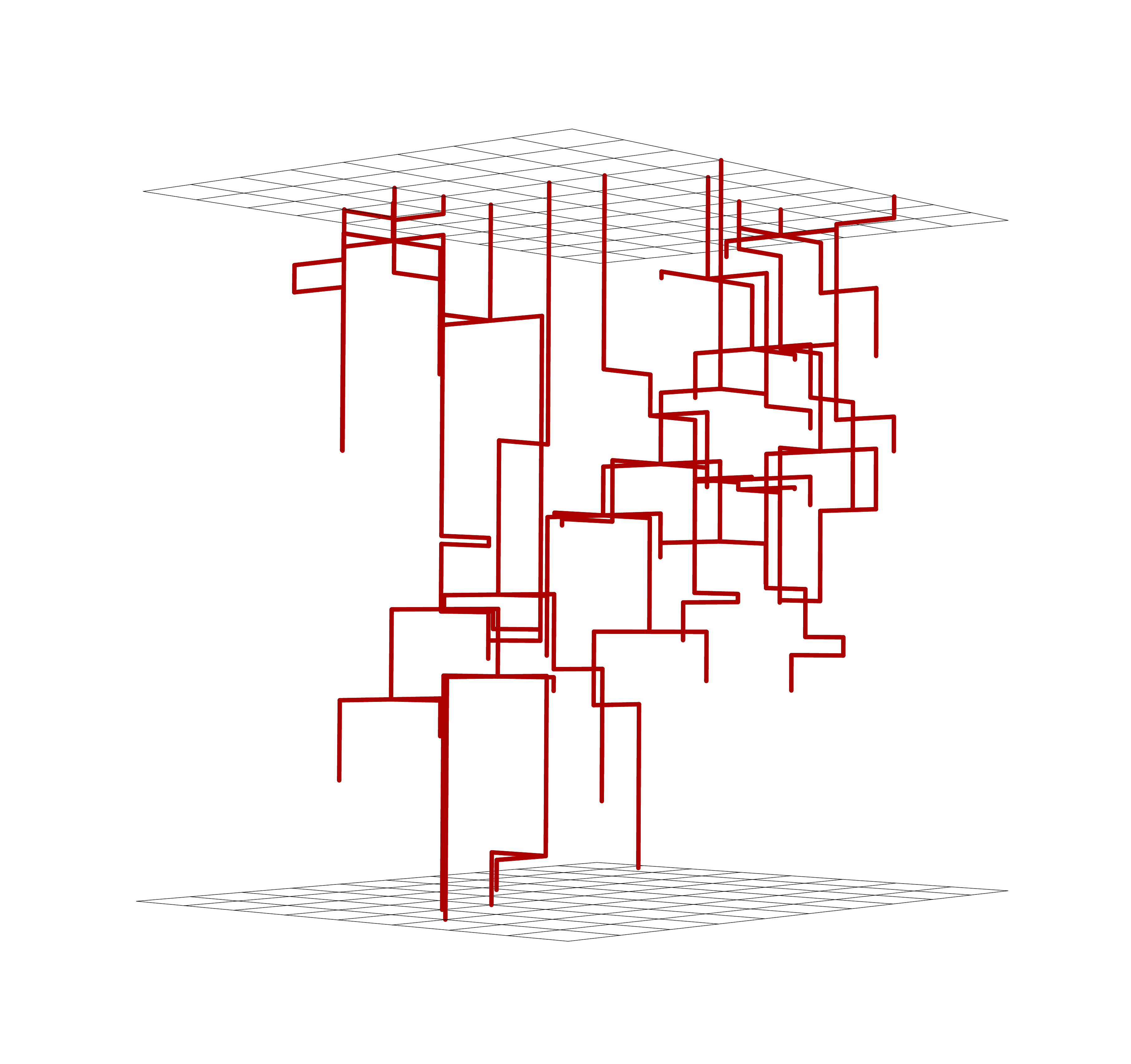}
\raisebox{0.05cm}{\includegraphics[width=.09\textwidth]{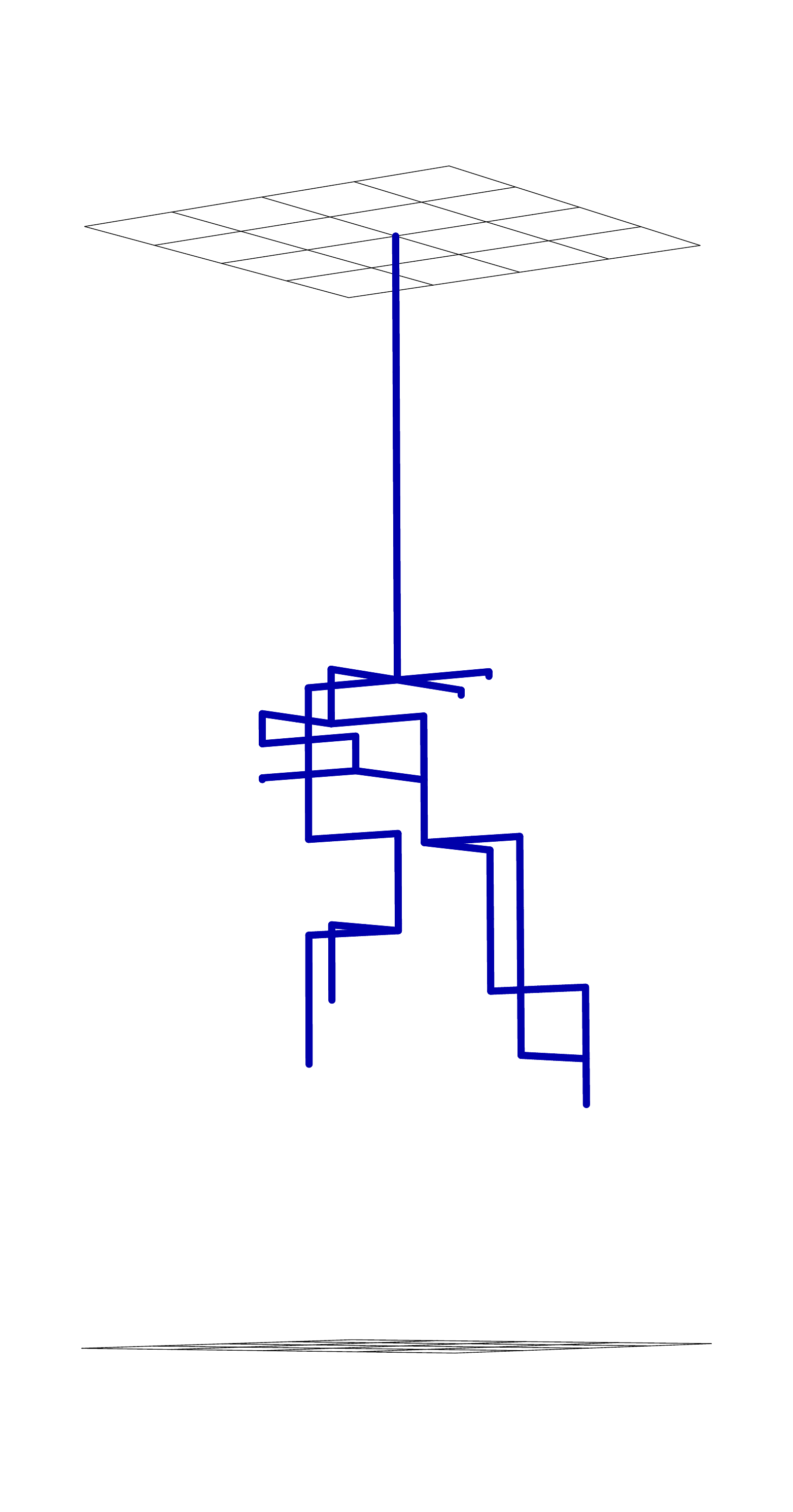}}
\includegraphics[width=.15\textwidth]{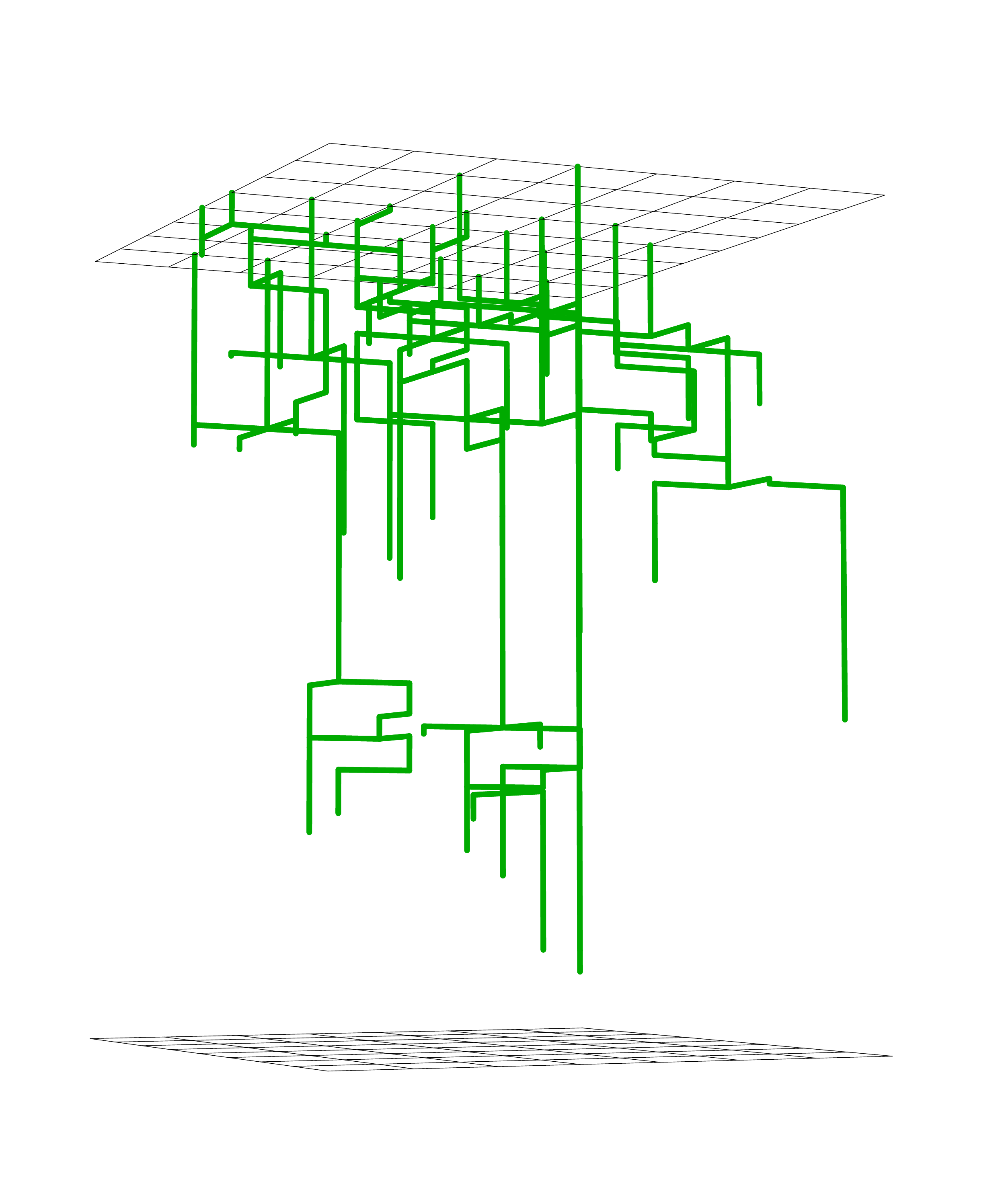}
\end{tabular}
\vspace{-0.4cm}
\caption{Information percolation clusters for the stochastic Ising model on two geometries: hyperbolic graph (left, showing largest 3 clusters of each type) and the lattice $\Z_{100}^2$ (right). A cluster is red if it survives to time 0, blue if it dies out and is the history of a single vertex, and green o/w.}
\label{fig:hyperbol}
\end{figure}

\begin{figure}[t]
\includegraphics[width=.475\textwidth]{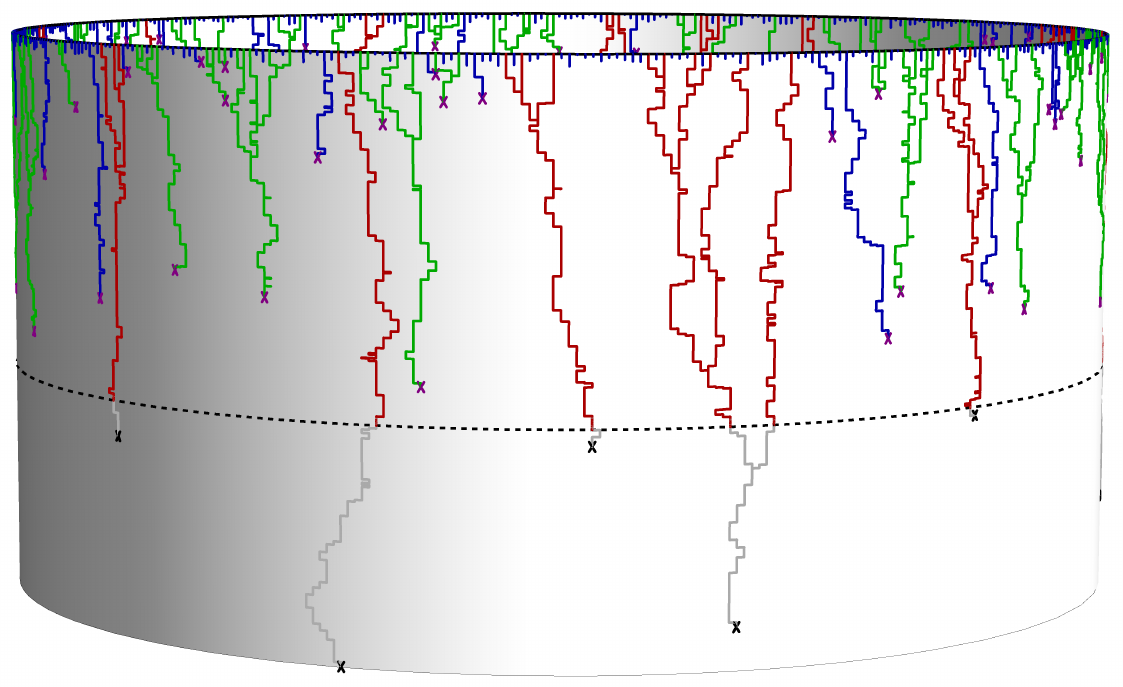}
\raisebox{-0.07cm}{
\includegraphics[width=.48\textwidth]{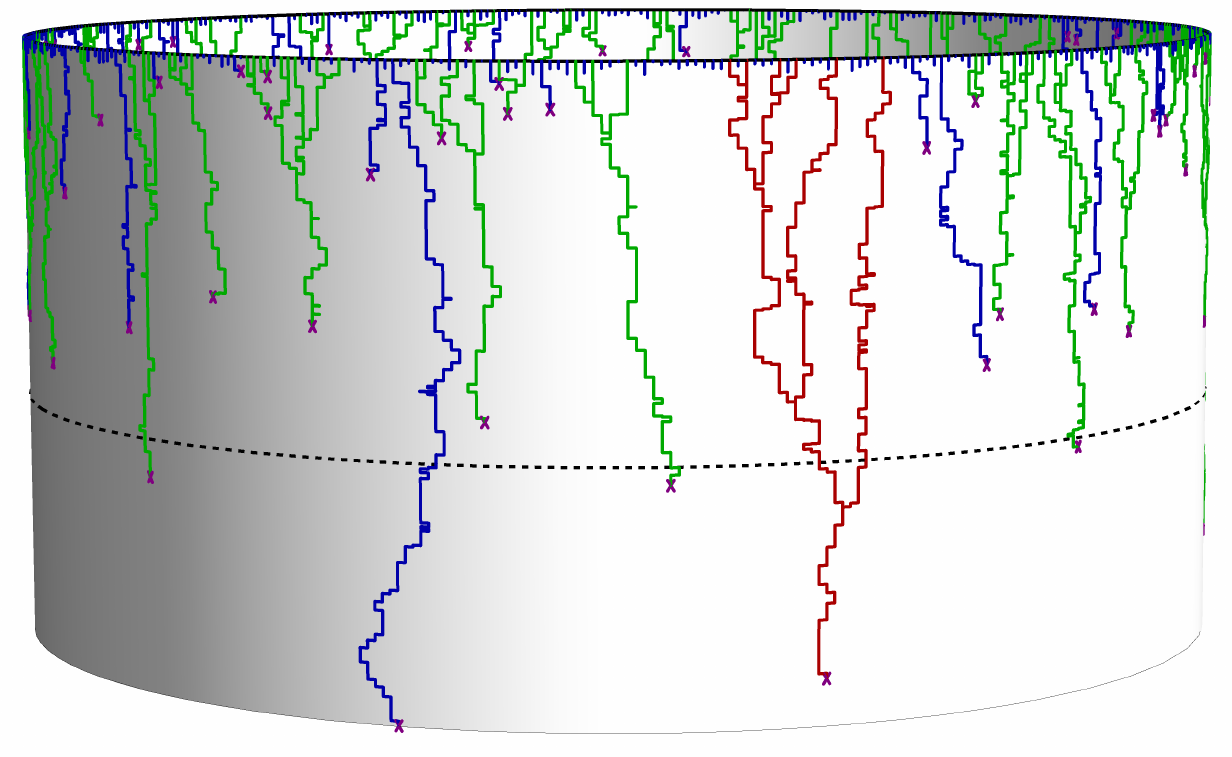}}
\vspace{-0.2cm}
\caption{Flavor of information percolation for analyzing random initial states in 1\textsc{d} Ising model:\\
 On the left, the standard framework (red clusters are those reaching $t=0$) for worst-case analysis. On the right, red clusters are redefined as those coalescing \emph{below} $t=0$ for the annealed analysis.}
\label{fig:clusters-1d}
\end{figure}

Even before requiring these powerful log-Sobolev inequalities, the restriction to sub-exponential growth rate automatically precluded the analysis of examples as basic as the Ising model on a binary tree at \emph{any} small $\beta>0$, or on an expander graph (e.g., a random regular graph), the hypercube, etc.

Here, using the framework of information percolation that we introduced in the companion paper~\cite{LS4}, we confirm that on any sequence of graphs with maximum degree $d$, cutoff indeed occurs whenever $\beta d$ is small enough, and with an $O(1)$-window (just as when $\beta=0$).
Furthermore, we analyze the effect of the initial state on the mixing time (e.g., a warm start of i.i.d.\ spins vs.\ the all-plus starting state).

\subsection{Results}
Our first result establishes that, on any geometry, at high enough temperature there is cutoff within an $O(1)$-window around the point
\begin{equation}\label{eq-t*-def}
\tcut = \inf\big\{\; t>0 \;:\; \mbox{$\sum_v \sm_t(v)^2 \leq 1$} \;\big\}\,,
\end{equation}
where $\sm_t(v)$ is the magnetization at a vertex $v\in V$ at time $t>0$, i.e.,
\begin{align}\label{eq-mag-v}\sm_t(v) = \E X_t^+(v)\,,\end{align}
with $X_t^+$ denoting the dynamics started from all-plus.
Note that on a transitive graph (such as $\Z_n^d$), the point $\tcut$ coincides with the time at which $\sum_v\sm_t(v)$ drops to a square-root of the volume, which has the intuitive interpretation that mixing occurs once the expected sum of spins in $X_t^+$ drops within the normal deviations in the Ising measure. However, it turns out that for general (non-transitive) geometries (such as trees) it is the sum of \emph{squared} magnetizations $\sum_v \sm_t(v)^2$ that governs the mixing.

\begin{maintheorem}\label{mainthm-gen}
There exist absolute constants $\kappa, C > 0$ such that the following holds. Let $G$ be a graph on $n$ vertices with maximum degree $d$. For any fixed $0<\epsilon<1$ and large enough $n$, the continuous-time Glauber dynamics for the Ising model on $G$ with inverse-temperature $0\leq\beta<\kappa/ d$ satisfies
\begin{align*}
 \tmix(1-\epsilon)&\geq \tcut - C\log(1/\epsilon)  \,,\\
   \tmix(\epsilon) &\leq \tcut + C\log(1/\epsilon)\,.
 \end{align*}
In particular, on any sequence of such graphs the dynamics has cutoff with an $O(1)$-window around $\tcut$.
\end{maintheorem}

\vspace{-0.1cm}

Apart from giving a first proof of cutoff for the Ising model on any tree / expander graph at $\beta>0$, note that the above theorem allows the maximum degree $d$ to depend on $n$ in any way, and so it applies, e.g., to the Ising model on the hypercube (with $d=\log_2 n$), a dense Erd\H{o}s-R\'enyi graph $\cG(n,\frac12)$, etc.

As mentioned above, the proof uses the new information percolation framework, which analyzes interactions between spins viewed as a percolation process in the space-time slab. As opposed to the application of this method in the companion paper~\cite{LS4} for the torus, various obstacles arise in the present setting due to the asymmetry between vertices and lack of amenability. Moreover, a na\"ive application of the method would require $\beta$ to be as small as about $d^{-d}$, and carrying it up to $\kappa/d$ (the correct dependence in $d$ up to the value of $\kappa$) required several novel ingredients, notably using a discrete Fourier expansion (see \S\ref{sec:fourier}) to prescribe update rules for the dynamics that would endow the resulting percolation clusters with a subcritical behavior.

Roughly put, the framework considers the dynamics at a designated time around $\tcut$, and for each site develops the history of updates that led to its final spin (tracing back branching to its neighbors). The resulting ``information percolation'' clusters in the space-time slab are then categorized into three types --- \red\ (those surviving to time zero and nontrivially depending on the initial state), \blue\ (those remaining which involve a unique ``ancestor'') and \green\ (all remaining clusters), as illustrated in Figure~\ref{fig:hyperbol}. The green clusters (which may exhibit complicated dependencies but are independent of the initial state) are taken out of the equation via conditioning, leaving behind a competition between blue clusters (whose ancestor vertices are i.i.d.\ uniform spins by symmetry) and red clusters. Controlling the latter, namely an exponential moment of their cumulative size, then establishes mixing.

Overall, the information percolation framework allows one to reduce challenging problems involving mixing and cutoff for the Ising model into simpler and tractable problems on subcritical percolation.

Furthermore, by analyzing not only on the size of the red clusters, but rather \emph{where} these hit the initial state at time zero, this framework opens the door to understanding the effect of the starting configuration on the mixing time (where sharp results on total-variation mixing for the Ising model were only applicable to worst-case starting states, usually via coupling techniques).

Our next result demonstrates this by comparing the worst-case mixing time (which is matched by the all-plus starting state up to an additive $O(1)$-term) with a typical starting configuration, and finally with the uniform starting configuration, i.e., each site is initialized by an independent uniform $\pm1$ spin.
Informally, we show that the uniform starting state is roughly at least twice faster compared to all-plus, but perhaps surprisingly, almost every deterministic starting state is about as slow as the worst one.

Formally, if $\mu_t^{(x_0)}$ is the distribution of the dynamics at time $t$ started from $x_0$ then $\tmix^{(x_0)}(\epsilon)$ is the minimal $t$ for which $\mu_t^{(x_0)}$ is within distance $\epsilon$ from equilibrium, and $\tmix^\uni(\epsilon)$ is the analogue for the average $2^{-n}\sum_{x_0}\mu_t^{(x_0)}$ (i.e., the \emph{annealed} version, as opposed to the \emph{quenched} $\tmix^{(X_0)}$ for a uniform $X_0$).
\begin{maintheorem}
  \label{mainthm-ann-que}
Consider continuous-time Glauber dynamics for the Ising model on an $n$-vertex graph $G$ with maximum degree at most some fixed $d>0$,
and define $\tcut$ as in~\eqref{eq-t*-def}.
For every $\epsilon>0$ there exists $\beta_0>0$ such that the following hold for any $0<\beta<\beta_0$ and any fixed $0<\alpha<1$ at large enough $n$.
\begin{enumerate}[\it1.]
  \item\label{it-ann} (Annealed) Uniform initial state:  $\tmix^\uni(\alpha) \leq (\frac12+\epsilon) \tcut$.
  \item\label{it-que} (Quenched) Deterministic initial state: $\tmix^{(x_0)}(\alpha) \geq (1-\epsilon) \tcut$ for almost every $x_0$, while $\tmix^{(+)}(\alpha)\sim \tcut$.
\end{enumerate}
\end{maintheorem}
The delicate part in the proof of the above theorem is comparing the distribution at time $t$ directly to the Ising measure. One often bypasses this point by coupling the distributions started at worst-case states; here, however, that would fail as we are analyzing the dynamics well before these distributions can couple with high probability. Instead (and as demonstrated in the companion paper for analyzing the effect of initial states in the 1\textsc{d} Ising model), we appeal to the \emph{Coupling From The Past} method~\cite{PW}.

Rather than developing the information percolation clusters until reaching time zero, we continue until time $-\infty$, letting all clusters eventually die. The beautiful Coupling From The Past argument implies that, if we ignore the initial state altogether, the final configuration would be a perfect simulation of the Ising measure. Thus, the natural coupling the information percolation clusters allows one to compare the dynamics with the Ising measure, simply by considering the effect of replacing the spins generated along the interval $(-\infty,0]$ by those of the initial state.

Specifically for the annealed analysis, even if a cluster survives to time zero (and beyond) it might still be perfectly coupled to the stationary measure, e.g., a singleton strand (and more generally, a blue cluster) would receive a uniform spin both from the Ising measure and from the random initial state. Hence, we modify the framework by redefining red clusters as those in which \emph{at least two branches} of the cluster reach time zero, then proceed to merge in the interval $(-\infty,0)$, as illustrated in Figure~\ref{fig:clusters-1d}. It is this factor of 2 that eventually transforms into the factor of $2-\epsilon$ improvement in the mixing time.

\subsection*{Organization}
The rest of this paper is organized as follows.
In~\S\ref{sec:framework} we give the formal definitions of the above described framework, including several modification needed here (e.g., custom update rules to be derived from a Fourier expansion) and two lemmas analyzing the information percolation clusters. In~\S\ref{sec:cutoff} we prove the cutoff result in Theorem~\ref{mainthm-gen} modulo these technical lemmas, which are proved in~\S\ref{sec:cluster-analysis}. The final section, \S\ref{sec:ann-que}, is devoted to the effect of the initial states on mixing and the proof of Theorem~\ref{mainthm-ann-que}.

\section{Information percolation for the Ising model}\label{sec:framework}

\subsection{Preliminaries}
In what follows we set up standard notation for analyzing the mixing of Glauber dynamics for the Ising model; see~\cites{LS1,LS4} and the references therein for additional information.

 \subsubsection*{Mixing time and cutoff}\label{sec:prelim-cutoff}
Let $(X_t)$ be an ergodic finite Markov chain with stationary measure $\pi$.
An important gauge in MCMC theory for measuring the convergence of a Markov chain to stationarity is its total-variation mixing time. Denoted $\tmix(\epsilon)$ for a precision parameter $0<\epsilon<1$, it is defined as
\[ \tmix(\epsilon) \deq \inf\Big\{t \;:\; \max_{x_0 \in \Omega} \| \P_{x_0}(X_t \in \cdot)- \pi\|_\tv \leq \epsilon \Big\}\,,\]
where here and in what follows $\P_{x_0}$ denotes the probability given $X_0=x_0$, and the total-variation distance $\|\cdot\|_\tv$ between two probability measures $\nu_1,\nu_2$
on a finite space $\Omega$ is given by
\[
\|\nu_1-\nu_2\|_\tv = \max_{A\subset \Omega} |\nu_1(A)-\nu_2(A)| = \tfrac12\sum_{\sigma\in\Omega} |\nu_1(\sigma)-\nu_2(\sigma)| \,,
\]
i.e., half the $L^1$-distance between the two measures.

Addressing the role of the parameter $\epsilon$, the cutoff phenomenon is essentially the case where the choice of any fixed $\epsilon$ does not affect the asymptotics of $\tmix(\epsilon)$ as the system size tends to infinity. Formally, a family of ergodic finite Markov chains $(X_t)$, indexed by an implicit parameter $n$, is said to exhibit \emph{cutoff} (a concept going back to the pioneering works~\cites{Aldous,DiSh}) iff the following sharp transition in its convergence to stationarity occurs:
\begin{equation}
\label{eq-cutoff-def}
\lim_{n\to\infty} \frac{\tmix(\epsilon)}{\tmix(1-\epsilon)}=1 \quad\mbox{ for any $0 < \epsilon < 1$}\,.
\end{equation}
That is, $\tmix(\alpha)=(1+o(1))\tmix(\beta)$ for any fixed $0<\alpha<\beta<1$. The \emph{cutoff window} addresses the rate of convergence in~\eqref{eq-cutoff-def}: a sequence $w_n = o\big(\tmix(e^{-1})\big)$ is a cutoff window if $\tmix(\epsilon) = \tmix(1-\epsilon) + O(w_n)$ holds for any $0<\epsilon<1$ with an implicit constant that may depend on $\epsilon$.
Equivalently, if $t_n$ and $w_n$ are sequences with $w_n =o(t_n)$, we say that a sequence of chains exhibits cutoff at $t_n$ with window $w_n$ if
\[\left\{\begin{array}
  {r}\displaystyle{\lim_{\gamma\to\infty} \liminf_{n\to\infty}
 \max_{x_0 \in \Omega} \| \P_{x_0}(X_{t_n-\gamma w_n} \in \cdot)- \pi\|_\tv
  = 1}\,,\\
  \displaystyle{\lim_{\gamma\to \infty} \limsup_{n\to\infty}
 \max_{x_0 \in \Omega} \| \P_{x_0}(X_{t_n+\gamma w_n} \in \cdot)- \pi\|_\tv
  = 0}\,.
\end{array}\right.\]
Verifying cutoff is often quite challenging, e.g., even for simple random walk on an expander graph, no examples were known prior to~\cites{LS2,LSexp} (while this had been conjectured for almost all such graphs), and to date there is no known transitive example (while conjectured to hold for all transitive expanders).

\subsubsection*{Glauber dynamics for the Ising model}\label{sec:prelim-ising}
Let $G$ be a finite graph $G$ with vertex-set $V$ and edge-set $E$. The Ising model on $G$ is a distribution over the set $\Omega=\{\pm1\}^V$
of possible configurations, each
corresponding to an assignment of plus/minus spins to the sites in $V$. The probability of $\sigma \in \Omega$ is given by
\begin{equation}
  \label{eq-Ising}
  \pi(\sigma)  = Z^{-1} e^{\beta \sum_{uv\in E} \sigma(u)\sigma(v) } \,,
\end{equation}
where the normalizer $Z=Z(\beta,h)$ is the partition function.
The parameter $\beta$ is the inverse-temperature, which we always to take to be non-negative (ferromagnetic). These definitions extend to infinite locally finite graphs (see, e.g.,~\cites{Liggett,Martinelli97}).

The Glauber dynamics for the Ising model (the \emph{Stochastic Ising} model) is a family of continuous-time Markov chains on the state space $\Omega$, reversible w.r.t.\ the Ising measure $\pi$, given by the generator
\begin{equation}
  \label{eq-Glauber-gen}
  (\mathscr{L}f)(\sigma)=\mbox{$\sum_{u} c(u,\sigma) \left(f(\sigma^u)-f(\sigma)\right)$}
\end{equation}
where $\sigma^u$ for $u\in V$ is the configuration $\sigma$ with the spin at the vertex $u$ flipped.
We will focus on the two most notable examples of Glauber dynamics, each having an intuitive and useful graphical interpretation where each site experiences updates via an associated i.i.d.\ rate-one Poisson clock:
\begin{compactenum}[(i)]
\item \emph{Metropolis}: flip $\sigma(u)$ if the new state $\sigma^u$ has a lower energy (i.e., $\pi(\sigma^u)\geq \pi(\sigma)$), otherwise perform the flip with probability $\pi(\sigma^u)/\pi(\sigma)$.
    This corresponds to $  c(u,\sigma) = \exp\left(2\beta\sigma(u)\sum_{v \sim u}\sigma(y)\right)  \;\wedge\; 1 $.
\item \emph{Heat-bath}:  erase $\sigma(u)$ and replace it with a sample from the conditional distribution given the spins at its neighboring sites. This corresponds to $c(u,\sigma) = 1/\left[1+ \exp\left(-2\beta\sigma(u)\sum_{v \sim u}\sigma(v)\right)\right]$.
\end{compactenum}
It is easy to verify that these chains are indeed ergodic and reversible w.r.t.\ the Ising distribution $\pi$.
Until recently, sharp mixing results for this dynamics were obtained in relatively few cases, with cutoff only known for the complete graph~\cites{DLP,LLP} prior to the works~\cites{LS1,LS3}.

\subsection{Red, green and blue information percolation clusters}
In what follows, we describe the basic setting of the framework, which will be enhanced in \S\ref{sec:exp-framework} to support the setting of Theorem~\ref{mainthm-gen} (where the underlying geometry may feature exponential growth rate and we are in the range $\beta<\kappa/d$).

The \emph{update sequence} of the Glauber dynamics along an interval $(t_0,t_1]$ is the set of tuples of the form $(J,U,\tau)$, where $t_0< \tau\leq t_1$ is the update time,
$J\in V$ is the site to be updated and $U$ is a uniform unit variable.
Given this update sequence, $X_{t_1}$ is a deterministic function of $X_{t_0}$, right-continuous w.r.t.\ $t_1$.

We call a given update $(J,U,\tau)$ an \emph{oblivious update} iff $U \leq \theta$ for
\begin{equation}\label{eq-def-theta}
  \theta = \theta_{\beta,d} := 1 - \tanh(\beta d)\,,
\end{equation}
since in that situation one can update the spin at $J$ to plus/minus with equal probability (that is, with probability $\theta/2$ each) independently of the spins at the neighbors of the vertex $J$, and a properly chosen rule for the case  $U > \theta$ legally extends this protocol to the Glauber dynamics.

Consider some designated target time $\tpluss$ for analyzing the spin distribution of the dynamics on $G$.
The \emph{update history} of $X_{\tpluss}(v)$ going back to time $t$, denoted $\sH_v(t)$, is a subset $A\times\{t\}$ of the space-time slab $V\times \{t\}$, such that one we can determine $X_{\tpluss}(v)$ from the update sequence and spin-set $X_t(A)$.
The most basic way of defining $\{ \sH_v(t) : 0\leq t \leq \tpluss\}$ is as follows:
\begin{compactitem}[\noindent$\bullet$]
  \item List the updates in reverse chronological order as $\{(J_i,U_i,t_i)\}_{i\geq 1}$ (i.e., $t_i > t_{i+1}$ for all $i$),
      and initialize the update history by $\sH_v(t)=\{v\}$ for all $t\in [t_1,\tpluss]$.

   \item In step $i\geq1$, process the update $(J_i, U_i, t_i)$ to determine $\sH_v(t)$ for $t\in[t_{i+1},t_i)$:
       \begin{itemize}[\noindent\ndash]
       \item If $J_i\notin \sH_v(t_{i})$ then the history is unchanged, i.e., $\sH_v(t) = \sH_v(t_i)$ for all $t\in[t_{i+1},t_i)$.
       \item If $J_i\in \sH_v(t_{i})$ but $U_i \leq \theta$ then $J_i$ is removed, i.e., $\sH_v(t) = \sH_v(t_i) \setminus \{J_i\}$ for all $t\in[t_{i+1},t_i)$.
   \item Otherwise, replace $J_i$ by its neighbors $N(J_i)$, i.e., $\sH_v(t) = \sH_v(t_{i}) \cup N(J_i) \setminus \{J_i\} $ for all $t\in[t_{i+1},t_i)$.
       \end{itemize}
\end{compactitem}
The \emph{information percolation clusters} are the connected components of the graph on the vertex set $V$ where $(u,v)$ is an edge if $\sH_u(t)\cap \sH_v(t) \neq \emptyset $ for some $t\geq0$.
Denote by $\cC_v$ the cluster containing $v\in V$.

We will also consider clusters in the context of the full space-time slab. The cluster of a point $(w,r) \in V\times[0,\tpluss]$, denoted $\cX_{w,r}$, is the connected component of
$\bigcup\{\sH_v(t): v\in V,\, 0\leq t \leq \tpluss\}$ that contains $(w,r$).
(Thus, the cluster $\cC_v$ is identified with the intersection of $\cX_{v,\tpluss}$ with the slab $V\times\{\tpluss\}$.)

For any $A\subset V$ we use the notation $ \sH_A(t) = \bigcup_{v\in A}\sH_v(t)$, as well as $\sH_A(t_1,t_2) =\bigcup_{t_1\leq t\leq t_2}\sH_A(t)$ (both cases describing subsets of $V$). Omitting the time subscript altogether would refer to the full time interval: $\sH_A := \sH_A(0,\tpluss)$, so that, for instance, if $\cC\subset V$ is a cluster then $\sH_\cC$ is the set of all vertices ever visited by this cluster. (By a slight abuse of notation, we may write $\sH_\cX := \sH_\cC$ with $\cC$ the cluster that $\cX$ identifies with (i.e., the cluster $\cC$ such that $\cC \times \{\tpluss\} =  (V\times\{\tpluss\})\cap\cX $).)
A final useful notation in this context is the collective history of $V\setminus A$, defined as
\[
\sH_A^- = \left\{\sH_v(t) : v\notin A\,,\,t\leq \tpluss\right\}
\,.\]
The clusters are classified into three classes (identifying for this purpose
$\cC_v$ and $\cX_{v,\tpluss}$) as follows:
\begin{compactitem}
\item A cluster $\cC$ is \red\ if, given the update sequence, its final state $X_{\tpluss}(\cC)$ is a nontrivial function of the initial configuration $X_0$; in particular,
     its history must survive to time zero ($\sH_\cC(0)\neq\emptyset$).
\item A cluster $\cC$ is \blue\ if it is a singleton --- i.e., $\cC =\{v\}$ for some $v\in V$ --- whose history does not survive to time zero ($\sH_v(0)=\emptyset$).
\item Every other cluster $\cC$ is \green.
\end{compactitem}
Note that if a cluster is blue then its single spin at time $\tpluss$ does not depend on the initial state $X_0$, and so, by symmetry, it is a uniform $\pm1$ spin. (While a green cluster is similarly independent of $X_0$, as multiple update histories intersect, the distribution of its spin set $X_{\tpluss}(\cC)$
may become quite nontrivial.)

Let $V_\red$ denote the union of the red clusters, and let $\sH_\red$ be the its collective history --- the union of $\sH_v(t)$ for all $v\in V_\red$ and $0\leq t\leq\tpluss$ (with analogous definitions for blue/green).

A beautiful short lemma of Miller and Peres~\cite{MP} shows that, if a measure $\mu$ on $\{\pm1\}^V$ is given by sampling a variable $R\subset V$ and using an arbitrary law for its spins and a product of Bernoulli($\frac12$) for $V\setminus R$, then the $L^2$-distance of $\mu$ from the uniform measure is at most $\E 2^{|R\cap R'|}-1$ for i.i.d.\ copies $R,R'$. (See Lemma~\ref{lem:MP} below; also see~\cite{LS4}*{Lemma 4.3} for a generalization of this to a product of general measures, which becomes imperative for the information percolation framework at $\beta$ near criticality.)
Applied to our setting, if we condition on $\sH_\green$ and look at the spins of $V\setminus V_\green$ then $V_\red$ can assume the role of the variable $R$, as the remaining blue clusters are a product of Bernoulli($\frac12$) variables.

In this conditional space, since the law of the spins of $V_\green$, albeit potentially complicated, is independent of the initial state, we can safely project the configurations
on $V\setminus V_\green$ without it increasing the total-variation distance between the distributions started at the two extreme states.
Hence, a sharp upper bound on worst-case mixing will follow by showing for this exponential moment
\begin{equation}
  \label{eq-exp-moment-bound}
  \E \left[ 2^{|V_\red \cap V_\red'|} \mid \sH_\green\right] \to 1\quad\mbox{ in probability as }n\to\infty\,,
\end{equation}
by coupling the distribution of the dynamics at time $\tpluss$ from any initial state to the uniform measure.
Finally, with the green clusters out of the picture by the conditioning (which has its own toll, forcing various updates along history so that no other cluster would intersect with those nor become green), we can bound the probability that a subset of sites would become a red cluster by its ratio with the probability of all sites being blue clusters. Being red entails connecting the subset in the space-time slab, hence the exponential decay needed for~\eqref{eq-exp-moment-bound}.

\subsection{Enhancements of the framework: custom update rules and modified last unit interval}\label{sec:exp-framework}
We will consider the information percolation clusters developed as above from the designated time
\[ \tpluss = \tcut + \scut\quad\mbox{ for }\quad\scut = C \log(1/\epsilon)\]
where $C>0$ will be specified later, and $\epsilon>0$ is the parameter for the mixing time. However, instead of the standard procedure of developing the history, where an update at $v$ either deletes it from the history (via an oblivious update) or replaces it by its set of neighbors $N(v)$, we will allow $v$ to be replaced (with varying probabilities) by any subset of its neighbors, in the following way.

Recall that an update of the form $(J,U,t)\in V\times [0,1]\times [0,\tpluss]$ results in replacing the spin at $J$ at time $t$ by some deterministic function $\Upsilon(x,U)$, where $x = \sum_{u\in N(J)} X_t(u)$.
A \emph{generalized} update rule observes updates of the form $(J,A,U,t)$ where $(J,U,t)$ is as before and the additional variable $A\subset [d]$ corresponds to a subset of the neighbors of vertex $J$. The new update rule exposes the spins $\{\sigma_1,\ldots,\sigma_{|A|}\}$ of these neighbors at time $t$, then generates the new spin at $J$ via $\Phi_A(\sigma_1,\ldots,\sigma_{|A|},U)$.

With this generalized update rule, one unfolds the update history of a vertex $\{ \sH_v(t) : 0\leq t \leq \tpluss\}$ as before, with the one difference that an update $(J_i,A_i,U_i,t_i)$ for which $J_i\in\sH_v(t_i)$ now results in $\sH_v(t) = \sH_v(t_{i}) \cup A_i \setminus \{J_i\} $ for all $t\in[t_{i+1},t_i)$.
The functions $\{\Phi_A : A\subset[d]\}$, as well as the probability distribution over the subsets $A\subset[d]$ to be exposed, will be derived from a discrete Fourier expansion of the original rule $\Upsilon$ (see Lemma~\ref{lem-p(k,r)}), so that the new update procedure would, one on hand, couple with the Glauber dynamics, and on the other,
  endow our percolation clusters with a subcritical behavior.

A final ingredient needed for coping with the arbitrary underlying geometry is a modification of the update history, denoted by $\hsH$:
in the modified version, every vertex $v\in V$ receives an (extra) update at time $\tpluss$, and no vertex is removed from the history along the unit interval $(\tpluss-1,\tpluss]$. (For a given update sequence, this operation can only increase any information percolation cluster, and forbidding vertices to die in the first unit interval will be useful in the context of conditioning on other clusters.) We will write $\hcC$, $\hcX$, as well as $\hsH_A(t)$ etc.\ for the corresponding notation w.r.t.\ the modified history $\hsH$.

\smallskip

We end this section with two results on the information percolation clusters --- Lemmas~\ref{lem-Psi} and~\ref{lem-todo} --- which will be central in the proof of Theorem~\ref{mainthm-gen}. The proofs of these lemmas are postponed to~\S\ref{sec:cluster-analysis}.

As explained following the definition of the three cluster types, at the heart of the matter is estimating an exponential moment of the size of the red clusters given $\sH_\green$, the joint history of all green clusters. To this end, we wish to bound the probability that a subset $A$ is a red cluster given $\sH_\green$.
Define
\begin{equation}
  \label{eq-Psi-def}
  \Psi_A = \sup_{\sH_A^-} \P\left(A\in \red \mid \sH_A^-\,,\,\{A\in\red\} \cup \{A \subset V_\blue\}\right) \,,
\end{equation}
noting that, towards estimating the probability of $A\in\red$, the effect of conditioning on $\sH_A^-$
amounts to requiring that $\sH_A$ must not intersect $\sH_A^-$.

\begin{lemma}\label{lem-Psi}
If $\beta < 1/(5d)$ then for any $A\subset V$ and $v\in A$,
\[
\Psi_A \leq 2^{|A|} \E\Big[\one_{\{A\subset \hcC_v\}}e^{\htau_v}\sum_w \one_{\{w\in \hsH_A(\tpluss-\htau_v,\tpluss)\}} \sm_{\tpluss}(w)\Big]
\]
where $\htau_v$ is the time it takes the history of $\hcC_v$
to first coalesce into a single point (if at all), i.e.,
\begin{equation}
  \label{eq-tauv-def}
  \htau_v=\min\Big\{ t\geq 1 : |\hsH_{\hcC_v}(\tpluss - t)| = 1\Big\} \wedge \tpluss\,.
\end{equation}
\end{lemma}
It is worthwhile noting in the context of the parameter $\htau_v$ that, when developing the update history backward in time, $\htau_v$ is not a stopping time, since $\hcC_v$ is affected by any potential coalescence points for $t < \tpluss- \htau_v$; instead, one can determine $\htau_v$ as soon as $\hsH_{\hcC_v}(t) = \emptyset$.
Also observe that $\htau_v=1$ iff $|\cC_v|=1$.
Finally, the coalescence point $w$ at time $t=\tpluss - \htau_v$ (when $t>0$) need not belong to $\hsH_v$ --- e.g., we may have $\hsH_v(t)=\emptyset$ while $w\in\hsH_u$ for some $u\neq v$ whose history intersected that of $v$ at time $t'>t$.

The subcritical nature of the information percolation clusters (prompted by our modified update functions $\Phi_A$) allows one to control exponential moments of the cluster sizes, as in the following lemma.
\begin{lemma}
  \label{lem-todo}
Fix $0<\eta<1$ and $\lambda>0$. There exist constants $\kappa,\gamma>0$ such that the following holds.
For any point $(w_0,t_0)$ in the space-time slab $ V \times (0,\tpluss]$, if $\beta < \kappa / d$ then
  \[ \E\left[ \exp\left( \eta \len(\hcX_{w_0,t_0}) + \lambda |\hsH_{\hcX_{w_0,t_0}}|\right)\right] < \gamma\,,\]
  where
  \[ \len(\hcX) = \sum_{u\in V} \int_0^{\tpluss} \one_{\{(u,t)\in\hcX\}}dt\,.\]
\end{lemma}
The above lemma, whose proof follows standard arguments from percolation theory, will be applied for absolute constants $\eta$ and $\lambda$ in the proof of Theorem~\ref{mainthm-gen}
(any $1/2<\eta<1$ and $\lambda>\log 8$ would do), leading to the absolute constant $\kappa$ in the statement of that theorem. The above formulation will be important in the context of Theorem~\ref{mainthm-ann-que}, where one requires $\eta$ that may be very close to $1$ (as a function of $\epsilon$ from the statement of that theorem) and $\lambda$ that depends on the maximum degree.

\section{Cutoff with constant window from a worst starting state}\label{sec:cutoff}
In this section we prove Theorem~\ref{mainthm-gen} via the framework defined in~\S\ref{sec:framework}. As is often the case in proofs of cutoff, the upper bound will require the lion's share of the efforts.
\subsection{Upper bound modulo Lemmas~\ref{lem-Psi} and~\ref{lem-todo}}
Define the coupling distance $\bar{d}_{\tv}(t)$ to be
\[ \bar{d}_\tv(t) = \max_{x_0,y_0} \left\| \P_{x_0}(X_t\in \cdot) - \P_{y_0} (X_t\in\cdot)\right\|_{\tv}\]
(so that $\frac12\bar{d}_\tv(t) \leq d_\tv(t) \leq \bar{d}_\tv(t)$), and observe that
\begin{align*}
 \bar{d}_\tv(t) &\leq  \E\Big[\max_{x_0,y_0}\left\|\P_{x_0}(X_t\in \cdot \mid \sH_\green) - \P_{y_0}(X_t\in\cdot\mid\sH_\green)\right\|_\tv \Big]\nonumber\\
&\leq \sup_{\sH_\green} \max_{x_0,y_0} \left\|\P_{x_0}(X_t(V \setminus V_\green) \in \cdot \mid \sH_\green) - \P_{y_0}(X_t(V \setminus V_\green) \in \cdot \mid \sH_\green)\right\|_\tv \,,
\end{align*}
where the first inequality follows by Jensen's Inequality and the second follows since $X_t(V_\green)$ is independent of the initial condition and so taking a projection onto $V\setminus V_\green$ does not change the total-variation distance between the distributions started at $x_0$ and $y_0$. Thus,
\begin{align}\label{eq-exp-moment-1}
 \bar{d}_\tv(t) &\leq 2 \sup_{\sH_\green} \max_{x_0} \left\|\P_{x_0}(X_t(V \setminus V_\green) \in \cdot \mid \sH_\green) - \nu_{V\setminus V_\green}\right\|_\tv\,,
\end{align}
where $\nu_A$ is the uniform measure on configurations on the sites in $A$. At this point we appeal to the exponential-moment bound of~\cite{MP}, whose short proof is included here for completeness.

\begin{lemma}[\cite{MP}]\label{lem:MP}
Let $\Omega=\{\pm1\}^V$ for a finite set $V$. For each $S\subset V$, let $\phi_S$ be a measure on $\{\pm1\}^{S}$.
Let $\nu$ be the uniform measure on $\Omega$, and let $\mu$ be the measure on $\Omega$ obtained by sampling a subset $S\subset V$ via some measure $\tilde{\mu}$, generating the spins of
 $S$ via $\phi_S$, and finally sampling $V \setminus S$ uniformly. Then
\[ \left\|\mu - \nu\right\|^2_{L^2(\nu)}   \leq
\E\left[2^{\left|S\cap S'\right|}\right]-1\,,\]
where the variables $S$ and $S'$ are i.i.d.\ with law $\tilde{\mu}$.
\end{lemma}
\begin{proof}
Write $n=|V|$, and let $x_S$ ($S\subset V$) denote the projection of $x$ onto $S$.  With this notation, by definition of the $L^2(\nu)$ metric (see, e.g.,~\cite{SaloffCoste2}) one has that $\|\mu-\nu\|^2_{L^2(\nu)} + 1 = \int |\mu/\nu - 1|^2 d\nu + 1 $ equals
  \begin{align*}
  \sum_{x\in\Omega} \frac{\mu^2(x)}{\nu(x)} 
  &=2^n \sum_{x\in\Omega}\sum_S \tilde{\mu}(S)\frac{\phi_S(x_S)}{2^{n-|S|}}\sum_{S'} \tilde{\mu}(S')\frac{\phi_{S'}(x_{S'})}{2^{n-|S'|}}
  \end{align*}
by the definition of $\mu$. Since $\sum_x \phi_S(x_S)\phi_{S'}(x_{S'})
\leq 2^{n-|S\cup S'|}$ it then follows that
\begin{equation*}
  \sum_{x\in\Omega}\frac{\mu^2(x)}{\nu(x)} \leq \sum_{S,S'}2^{|S|+|S'|-|S\cup S'|}\tilde\mu(S)\tilde\mu(S')
  =\sum _{S,S'} 2^{\left|S\cap S'\right|} \tilde{\mu}(S)\tilde{\mu}(S')
  \,.\qedhere
  \end{equation*}
\end{proof}
\begin{remark}
In the special case where the distribution $\phi_S$ is a point-mass on all-plus for every $S$, the single inequality in the above proof is an equality (since then $\sum_x\phi_S(x_S)\phi_{S'}(x_{S'})=\#\{x : x_{S\cup S'}\equiv 1\}$) and so
in that situation the $L^2$-distance $\|\mu - \nu\|^2_{L^2(\nu)}$ is precisely equal to $
\E\big[2^{|S\cap S'|}\big]-1$.

For example, consider Glauber dynamics for an $n$-vertex graph at $\beta=0$ (i.e., continuous-time lazy random walk on the hypercube $\{\pm1\}^n$) starting (say) from all-plus, and let $S$ be the set of coordinates which were not updated: here $\P(v\in S)=e^{-t}$ at time $t$, and $\|\P(X_t^+\in\cdot) - \nu\|^2_{L^2(\nu)} =
(1+e^{-2t})^n-1.$
\end{remark}

Applying the above lemma to the right-hand side of~\eqref{eq-exp-moment-1},
while recalling that any two measures $\mu$ and $\nu$ on a finite probability space satisfy $\|\mu-\nu\|_\tv  =\frac12 \|\mu-\nu\|_{L^1(\nu)} \leq \frac12 \|\mu-\nu\|_{L^2(\nu)}$,
we find that
\begin{align}\label{eq-exp-moment-2}
\bar{d}_\tv(\tpluss) \leq \Big( \sup_{\sH_\green} \E
\left[2^{\left|V_\red \cap V_{\red'}\right|} \;\big|\; \sH_\green\right] - 1 \Big)^{1/2} \,,
\end{align}
where $V_\red$ and $V_{\red'}$ are i.i.d.\ copies of the variable $\bigcup\{v \in V : \cC_v \in \red\}$.

Let $\{ Y_{A,A'} : A,A'\subset V\}$ be a family of independent indicators satisfying
\begin{equation}
  \label{eq-YAA'-def}
  \P(Y_{A,A'}=1) = \Psi_{A}\Psi_{A'}\quad\mbox{ for any $A,A'\subset V$}\,.
\end{equation}
We claim that it is possible to couple the conditional distribution of $(V_\red,V_{\red'})$ given $\sH_\green$ to the variables $Y_{A,A'}$ in such a way that
\[ \left|V_\red \cap V_{\red'}\right| \preceq \sum_{A\cap A'\neq \emptyset} |A\cup A'| Y_{A,A'}\,.\]
To do so,
let $\{(A_l,A'_l)\}_{l\geq 1}$ denote all pairs of intersecting subsets ($A,A'\subset V \setminus V_\green$ with $A\cap A'\neq\emptyset $) arbitrarily ordered, associate each pair with a variable $R_l$ initially set to 0, then process these in order:
 \begin{itemize}
   \item If $(A_l,A'_l)$ is such that, for some $j<l$, one has $R_j=1$ and either $A_j\cap A_l\neq\emptyset $ or $A'_j\cap A'_l\neq\emptyset$, then skip this pair (keeping $R_l=0$).
   \item Otherwise, set $R_l$ to the indicator of $\{A_l\in\red,\, A'_l\in\red'\}$.
 \end{itemize}
The claim is that $\P(R_l=1 \mid \cF_{l-1}) \leq \P(Y_{A_l,A'_l}=1)$ for all $l$, where
$\cF_l$ denotes the natural filtration associated to the above process. Indeed, consider some $(A_l,A'_l)$ for which we are about to set $R_l$ to the value of $\one_{\{A_l\in\red,\,A'_l\in\red'\}}$, and take any $A_j$ ($j<l$) such that $A_j\cap A_l\neq\emptyset$ and $\one_{\{\cC_j\in\red,\,\cC'_j\in\red'\}}$ was revealed (and necessarily found to be zero, by definition of the above process). The supremum over $\sH_{A_l}^-$ in the definition of $\Psi_{A_l}$ implies that we need only consider the information $\cF_{l-1}$ offers on $\sH_{A_l}$:
\begin{compactitem}[$\bullet$]
  \item
  If $A_j \cap A_l \neq A_l$ then
  the event $\{A_j\in \red\}$ does not intersect the event $\{A_l \in \red\} \cup \{A_l\subset V_\blue\}$
  (on which we condition in $\Psi_{A_l}$) as it requires $A_j$ to be a full red cluster
(so a strict subset of $A_j$ cannot belong to a separate red cluster, nor can it contain any blue singleton).
\item If $A_j = A_l$, conditioning on $\{A_j\in \red,\, A'_j\in\red'\}^c$ will not increase the probability of $\{A_l\in\red\}$.
\end{compactitem}
Either way, $\P(A_l \in \red \mid \cF_{l-1}) \leq \Psi_{A_l}$.
Similarly, $\P(A'_l\in \red'\mid \cF_{l-1},\,\one_{\{A_l\in \red\}})\leq \Psi_{A_l'}$, and together these inequalities support the desired coupling, since if $v\in V_\red \cap V_{\red'}$ then there is some $l$ for which $v\in A_l \cup A_{l'}$ and $A_l\in\red$, $A_{l'}\in\red'$, in which case every $A_j$ intersecting $A_l$ nontrivially will receive $R_j=0$ (it cannot be red) and the first $j$ with $A_j=A_l$ to receive $R_j=1$ will account for $v$ in $ A_j\cup A'_j$.

Relaxing $|A\cup A'|$ into $|A|+|A'|$ (which will be convenient for factorization), we get \begin{align*}
\sup_{\sH_\green}\E\left[2^{|V_\red \cap V_{\red'}|} \;\big|\; \sH_\green\right] &\leq  \E\left[ 2^{\sum_{A\cap A'\neq\emptyset} (|A|+|A'|) Y_{A,A'}}\right] = \prod_{A\cap A'\neq \emptyset} \E\left[ 2^{(|A|+|A'|)Y_{A,A'}}\right]\,,
\end{align*}
with the equality due to the independence of the $Y_{A,A'}$'s. By the definition of these indicators in~\eqref{eq-YAA'-def}, this last expression is at most
\begin{align*}
& \prod_v \!\prod_{\substack{A,A'\\ v\in A\cap A'}} \!\left(\big(2^{|A|+|A'|}-1\big) \Psi_{A}\Psi_{A'}+1\right)
\leq \exp\bigg[ \sum_v \bigg(\sum_{A \ni v} 2^{|A|} \Psi_{A}\bigg)^2\bigg]\,,
\end{align*}
and so, revisiting~\eqref{eq-exp-moment-2}, we conclude that
\begin{equation}\label{eq-req-exp-bound}
 \bar{d}_\tv(\tpluss)^2 \leq  \bigg( \exp\bigg[ \sum_v \bigg(\sum_{A \ni v} 2^{|A|} \Psi_{A}\bigg)^2\bigg] - 1\bigg) \wedge 1 \leq 2 \sum_v \bigg(\sum_{A \ni v} 2^{|A|} \Psi_{A}\bigg)^2\,,
 \end{equation}
where we used that $e^x-1\leq 2x$ for $x\in[0,1]$. We have thus reduced the upper bound in Theorem~\ref{mainthm-gen} into showing that
the right-hand of~\eqref{eq-req-exp-bound} is at most $\epsilon$
if $\scut= C \log(1/\epsilon)$ for some large enough $C=C(\beta)$.

Plugging the bound on $\Psi_A$ from Lemma~\ref{lem-Psi} shows that the sum in the right-hand of~\eqref{eq-req-exp-bound} is at most
\[\sum_v \Bigg(\sum_{A\ni v} 4^{|A|}\E\bigg[\one_{\{A\subset \hcC_v\}}
e^{\htau_v}\sum_w
\one_{\{w\in \hsH_A(\tpluss-\htau_v,\tpluss)\}} \sm_{\tpluss}(w)\bigg]\Bigg)^2\,.\]
In each of the two sums over $A\ni v$ we can specify the size of $\hcC_v$, and then relax $\{w\in \hsH_A(\tpluss-\htau_v,\tpluss)\}$ into $\{w\in\hsH_{\hcC_v}\}$ (thus permitting all $2^{|\hcC_v|}$ subsets to play the role of $A$); thus, the last display is at most
\begin{align}
   \sum_v \sum_{k,k'} \sum_{w,w'} 8^k &\E\left[\one\left\{|\hcC_v|=k,\,
   w\in\hsH_{\hcC_v}\right\} e^{\htau_v} \sm_{\tpluss}(w)\right]\nonumber\\
   8^{k'}&\E\left[\one\left\{|\hcC_v|=k',\,
   w'\in\hsH_{\hcC_v}\right\} e^{\htau_v} \sm_{\tpluss}(w')\right]\,.\label{eq-sum-k-k'}
 \end{align}
 Denoting the indicators above by $\Xi(v,w,k)$ and $\Xi(v,w',k')$ respectively, and using the fact that
 \[ \sum_{w,w'} \sm_t(w) \sm_t(w') \leq \frac12 \sum_{w,w'}\left(\sm_t(w)^2 + \sm_{t}(w')^2\right) = \sum_{w,w'} \sm_t(w)^2\]
 in~\eqref{eq-sum-k-k'} culminates in the following bound on sum in the right-hand of~\eqref{eq-req-exp-bound}:
\begin{align}
  \sum_v \bigg(\sum_{A \ni v} 2^{|A|} \Psi_{A}\bigg)^2\leq &\sum_w \sm_{\tpluss}(w)^2 \sum_k \sum_v \E\left[8^k \Xi(v,w,k) e^{\htau_v} \right] \sum_{k'} \E\bigg[e^{\htau_v}\sum_{w'}8^{k'} \Xi(v,w',k')  \bigg]\,.\label{eq-sum-k-k'-2}
 \end{align}
For the summation over $k'$ in~\eqref{eq-sum-k-k'-2}, we combine the facts that $\htau_v \leq \frac12 \len(\hsH_{\hcC_v}(\tpluss-\htau_v,\tpluss))+1 \leq \frac12 \len(\hsH_{\hcC_v})+1$ (either $|\hcC_v|=1$ and then $\htau_v=1$, or $|\hcC_v|\geq 2$ whence at least two strands survive for a period of $\htau_v$), that at most $|\hsH_{\hcC_v}|$ choices for $w'$ support $\Xi(v,w',k)=1$ and that $\sum_{k'}\Xi(v,w',k) \leq 1$, to get
\begin{align}
   \sum_{k'} \E\bigg[e^{\htau_v}\sum_{w'}8^{k'} \Xi(v,w',k')  \bigg]
\leq \E  \bigg[|\hsH_{\hcC_v}|\, 8^{|\hsH_{\hcC_v}|}e^{\frac12 \len\left(\hsH_{\hcC_v}\right)+1}\bigg]\leq \gamma_1\label{eq-sum-k-k'-last-subsum}
\end{align}
for some absolute constant $\gamma_1>0$, where the last inequality applied Lemma~\ref{lem-todo}.

Next, to treat the summation over $k$ in~\eqref{eq-sum-k-k'-2}, recall that
$\hcX_{w,r}$ for $(w,r)\in V\times[0,\tpluss]$ is the information percolation cluster containing the point $(w,r)$ in the space-time slab (i.e., the cluster is exposed from time $r$ instead of time $\tpluss$ and the process of developing it moves both forward and backward in time). Further write $\hcX^+_{w,r} = \lim_{t\to r^+} \hcX_{w,t}$ and $\hcX^-_{w,r}=\lim_{t\to r^-}\hcX_{w,t}$.

We claim that whenever $\Xi(v,w,k)=1$, necessarily $v\in \hcX^-_{w,r}$ for some $r\in \Pi_w$, where $\Pi_w$ records the update times for the vertex $w$ (always including $\tpluss$, by definition of $\hsH$).
Indeed, if $w\in\hsH(\tpluss-\htau_v,\tpluss)$ then by definition we can find some $q\in(\tpluss-\htau_v,\tpluss)$ such that $(w,q)$ shares the same information percolation cluster as $(v,\tpluss)$. Furthermore, if $r$ is the earliest update of $w$ after time $q$ then the cluster of $(w,t)$ for any $t\in(q,r)$ will contain $(w,q)$, and thus $(v,\tpluss)$ as-well. (It is for this reason that we addressed $\hcX^-_{w,r}$, in case the update at $(w,r)$ should cut its information percolation cluster from $(w,q)$.)
For that $r$, we further have $\tpluss-r\leq \htau_v \leq \frac12 \len(\hcX_{w,r}^-)+1$, and so
\begin{align*}
   \sum_k \sum_v\E\left[8^k \Xi(v,w,k) e^{\htau_v} \right] &\leq
  \E\left[\sum_{r\in \Pi_w} |\hsH_{\hcX_{w,r}^-}|\,8^{|\hsH_{\hcX_{w,r}^-}|} e^{\frac12\len\big(\hcX_{w,r}^-\big)}\one_{\big\{\frac12\len(\hcX_{w,r}^-)
  \geq\tpluss-r-1\big\}} \right]\\
    &\leq \E\left[\sum_{r\in \Pi_w} |\hsH_{\hcX_{w,r}^-}|\, 8^{|\hsH_{\hcX_{w,r}^-}|} e^{\frac34\len\big(\hcX_{w,r}^-\big) - \frac12(\tpluss-r-1)} \right]\,,
  \end{align*}
which, recalling that $\Pi_w$ is the union of $\{\tpluss\}$ and a rate-1 Poisson process, is at most
  \begin{align*}
  &
  \E\left[|\hsH_{\hcC_w}|\, 8^{|\hsH_{\hcC_w}|}\,e^{\frac34\len(\hsH_{\hcC_w})+\frac12}\right] +
  \int_{0}^{\tpluss}
  \E\left[\sum_{r\in \Pi_w} |\hsH_{\hcX_{w,r}^-}| \,8^{|\hsH_{\hcX_{w,r}^-}|}\, e^{\frac34\len(\hcX_{w,r}^-) - \frac12(\tpluss-r-1)} \;\bigg|\; r\in \Pi_w\right] dr \\
  \qquad&\leq
  \sqrt{e} \gamma_2 \left[ 1 + \int_{0}^{\tpluss}e^{- \frac12(\tpluss-r)} dr\right] \leq 5 \gamma_2
\end{align*}
for some absolute constant $\gamma_2>0$, using Lemma~\ref{lem-todo} (with $\gamma_2$ from that lemma) for the first inequality.

Substituting the last two displays together with~\eqref{eq-sum-k-k'-last-subsum} in~\eqref{eq-sum-k-k'-2}, while recalling~\eqref{eq-req-exp-bound}, finally gives
\begin{align}
 \bar{d}_\tv(\tpluss)^2\leq 10\gamma_1\gamma_2 \sum_w \sm_{\tpluss}(w)^2 \,.\label{eq-req-exp-bound2}
\end{align}
The proof will be concluded with the help of the next simple claim that establishes a submultiplicative bound for the second moment of the magnetization.
\begin{claim}\label{clm:MagnetDecay}
For any $t,s>0$ we have
\[
e^{-2s}\leq
\frac{\sum_w \sm_{t+s}(w)^2}{\sum_w\sm_t(w)^2} \leq e^{-2(1-\beta d) s }
\,.\]
\end{claim}
\begin{proof}
The lower bound follows from the straightforward fact that $\sm_{t+s}(w) \geq e^{-s} \sm_t(w)$ for any $s,t>0$ and $w$, since the probability of observing no updates to $w$ along the interval $(t,t+s)$ (thus maintaining the magnetization without a change) is $e^{-s}$. It therefore remains to prove the upper bound.

By expanding the probability of $1$ in an update, which is $\frac12 + \frac12\tanh(\beta \sigma)$ given a sum of neighbors of $\sigma$, and using the fact that $\frac{d}{dx}\tanh(x) \leq 1$ for any $x\in\R$, we have that upon updating $v$
\begin{align*} \P(X_t^+(v)=1)- \P(X_t^-(v)=1) &= \frac12\E\left[\tanh(\beta \sum_{w\sim v} X_t^+(w))-\tanh(\beta\sum_{w\sim v}X_t^-(w))\right] \\
& \leq \frac\beta2 \E\left[\sum_{w\sim v} X_t^+(w)-X_t^-(w)\right] = \beta\sum_{w\sim v} \sm_t(w)\,,
\end{align*}
and so $\frac{d}{dt}\sm_t(v) \leq \beta \sum_{w\sim v}\sm_t(w) - \sm_t(v)$. Hence,
\begin{align*}
  \frac{d}{dt} \sum_v \sm_t(v)^2 &= 2 \sum_v \sm_t(v) \frac{d}{dt}\sm_t(v) \leq -2\sum_v \sm_t(v)^2 + 2\beta \sum_v \sm_t(v) \sum_{w\sim v}\sm_t(w)\,,
  \end{align*}
and using $\sm_t(v)\sm_t(w) \leq \frac12(\sm_t(v)^2 + \sm_t(w)^2)$ it follows that
\begin{align*}
  \frac{d}{dt} \sum_v \sm_t(v)^2 &\leq -2\left(1 - \beta d\right) \sum_v \sm_t(v)^2\,,  \end{align*}
which implies the desired upper bound.
\end{proof}
Recalling that $\tpluss = t_\sm +\scut$, we apply the above claim
for $t=t_\sm$ (at which point $\sum_{w}\sm_{t_\sm}(w)^2=1$ by definition) and $s=\scut$ to find that $\sum_w \sm_{\tpluss}(w)^2 \leq \exp(-2(1-\beta d ) \scut) \leq \exp(-\scut)$, with the last inequality via $\beta d \leq \frac12$.
By~\eqref{eq-req-exp-bound2} (keeping in mind that $\gamma_1$ and $\gamma_2$ are absolute constants) this implies that $ \bar{d}_\tv(\tpluss) \leq \epsilon$ if we take $\scut \geq C \log(1/\epsilon)$
for some absolute constant $C>0$, as required.
\qed

\subsection{Lower Bound}

We now estimate the correlation of two vertices at an arbitrary time.
\begin{claim}\label{clm:Correlation}
There exist absolute constants $\kappa,\gamma>0$ such that, for any initial state, if $\beta<\kappa/d$ then
\[
\sum_u \Cov(X_{t}(u), X_t(v)) \leq \gamma \quad\mbox{ for any $t>0$ and $v\in V$}\,.
\]
\end{claim}
\begin{proof}
Let $X_{t}'$ and $X_{t}''$ be two independent copies of the dynamics.  By exploring the histories of the support we may couple $X_t$ with $X_{t}'$ and $X_{t}''$ so that, on the event $\{ u\notin \cC_v\}$,
the history of $u$ in $X_t$ is equal to the history of $u$ in $X_{t}'$ and the history of $v$ in $X_t$ is equal to the history of $v$ in $X_{t}''$.  Hence,
\begin{align*}
\E\left[X_t(u) X_t(v)\right] &= \E\left[ X_t'(u) X_t''(v) + \big(X_t(u) X_t(v) - X_t'(u) X_t''(v)\big)\one_{\{u\in \cC_v\}}\right]\\
&\leq \E\left[ X_t'(u)\right]\E\left[ X_t''(v)\right] + 2\P(u\in \cC_v)\,.
\end{align*}
It follows that $\Cov(X_t(u), X_t(v)) \leq 2\P(u\in \cC_v)\leq 2\P(u\in \hcC_v)$, and so
\[
\sum_u \Cov(X_{t}(u), X_t(v)) \leq 2 \E |\hcC_v| \leq \gamma\,,
\]
with the final equality thanks to Lemma~\ref{lem-todo}.
\end{proof}

We are now ready to prove the lower bound on the mixing time in Theorem~\ref{mainthm-gen}.
To this end, we use the magnetization to generate a distinguishing statistic
at time $\tminuss = \tcut-\scut$, given by
\[ f(\sigma) = \sum_{v\in V} \sm_{\tminuss}(v) \sigma(v)\,.\]
  Putting $Y = f\big(X^+_{\tminuss}\big)$ for the dynamics started from all-plus and $Y' = f(\sigma)$ with $\sigma$ drawn from the Ising distribution $\pi$, we combine
Claim~\ref{clm:MagnetDecay} with the fact that $\sum_v \sm_{\tcut}(v)^2 = 1$ (by definition) to get
\begin{align}
  \label{eq-E[Y]}
\E Y = \sum_{v} \sm_{\tminuss}(v)^2 \geq e^{2(1-\beta d) \scut} \sum_{v}\sm_{\tcut}(v)^2 = e^{2(1-\beta d)\scut} \geq e^{\scut}
\end{align}
(the last inequality using $\beta d \leq \frac12$),
whereas $\E Y' = 0$ (as $\E[\sigma(v)]=0$ for any $v$).

For the variance estimate, observe that
\begin{align*}
\var\left(Y\right) &= \sum_{u,v} \sm_{\tminuss}(u)\sm_{\tminuss}(v) \Cov\left(X^+_{\tminuss}(u),X^+_{\tminuss}(v)\right) \leq \frac12\sum_{u,v}\left(\sm_{\tminuss}(u)^2 +\sm_{\tminuss}(v)^2\right)\Cov\left(X^+_{\tminuss}(u),X^+_{\tminuss}(v)\right) \\
& \leq \gamma \sum_{v}\sm_{\tminuss}(v)^2 = \gamma \, \E Y\,,
\end{align*}
using Claim~\ref{clm:Correlation} for the inequality in the last line.
Furthermore, since the law of $X_t$ converges as $t\to\infty$ to that of $\sigma$, for any $v\in V$ we have
\[ \sum_u \Cov(\sigma(u),\sigma(v)) = \lim_{t\to\infty} \sum_u \Cov\left(X_t(u),X_t(v)\right) \leq \gamma\,, \]
and so the same calculation in the above estimate for $\var(Y)$ shows that
\[
\var\left(Y'\right) \leq \gamma\, \E Y\,.
\]

Altogether, by Chebyshev's inequality,
\[
\P\left(Y \geq \tfrac23 \E Y\right) \geq 1 -  9\gamma /\E Y \,,
\]
whereas
\[
\P\left( Y' \leq \tfrac13 \E Y\right) \geq 1-  9\gamma /\E Y\,.
\]
Recalling~\eqref{eq-E[Y]}, the expression $9\gamma/\E Y$ can be made less than $\epsilon/2$ by choosing $\scut \geq C \log(1/\epsilon)$ for some absolute constant $C>0$, thus concluding the proof of the lower bound.
\qed

\section{Analysis of percolation clusters}\label{sec:cluster-analysis}

\subsection{Red clusters: Proof of Lemma~\ref{lem-Psi}}
As we condition on the fact that either $A\in \red$ or $A\subset V_\blue$, as well as on the collective history of every $v\notin A$,
the history of the vertices of $A$ must avoid $\sH_A^-$ --- an event that we mark as $\cM$ --- and then give rise to blue clusters or a single red one
(we are interested in bounding the probability of the latter).
To analyze the probability of $\cM$, for each $u\in A$ we look at the latest time at which $\sH_A^-$ contains it ($u$ is ``undercut'' by $\sH_{A^-}$), that is,
\[ s_u = s_u(\sH_A^-) = \max\left\{ s \;:\; u\in\sH_{V\setminus A}(s)\right\}\,,\]
and focus our attention on the vertices that are undercut in the unit interval $(\tpluss-1,\tpluss]$ (which is the first unit interval to be exposed when developing $\sH_A$), writing
\[ A' = \left\{ u \in A \;:\; s_u > \tpluss-1\right\}\,,\]
and we denote by $\cU$ the event that every $u\in A'$ received an update in the interval $(s_u,\tpluss]$, which is of course a necessary condition for $\cM$ (so as to avoid the scenario where $u\in \sH_u(s_u)$ and intersects $\sH_A^-$ at that point).
With this in mind, for any $A\subset V$ and $\sH_A^-$ we have
\[\P\left(A\in\red\mid\sH_A^-\,,\{A\in\red\}\cup\{A\subset V_\blue\}\right)=\frac{\P(A\in\red\,,\,\cM\mid \cU)}{\P(\{A\in\red\}\cup\{A\subset V_\blue\}\,,\, \cM\mid \cU)}\,.\]
The numerator is at most $\P(A\in\red\mid\cU)$, while the denominator can be bounded from below by the probability that, in the space conditioned on $\cU$, the last update to each $u\in A$
occurs in the interval $(s_u \vee \tpluss-1,\tpluss]$ and it is oblivious (implying that its history amounts to the singleton $\{u\}$ dying out prior to being possibly undercut by $\sH_A^-$, and so $u\in V_\blue$).
Hence,
\[ \P(\{A\in\red\}\cup\{A\subset V_\blue\}\,,\, \cM\mid \cU) \geq \theta^{|A|}(1-1/e)^{|A\setminus A'|} > 2^{-|A|}\,,\]
where the term $\theta^{|A|}$ accounts for the probability that the latest most update is oblivious, the factor $(1-1/e)$ requires an update for vertices of $A\setminus A'$ (whose update in the last unit interval was not guaranteed by $\cU$), and the last inequality used that $\theta(1-1/e) \geq  (1-\tanh(\frac15))(1-1/e) > \frac12$ by our assumption on $\beta$ and the definition of $\theta$ in~\eqref{eq-def-theta}.
Overall,  we find that
\begin{align}
  \label{eq-psi-bound-given-U}
\P\left(A\in\red\mid\sH_A^-\,,\{A\in\red\}\cup\{A\subset V_\blue\}\right) \leq 2^{|A|} \P(A\in\red \mid \cU)\,.
\end{align}
Recall that in order for $A$ to form a complete red cluster, the update histories $\{\sH_u : u\in A\}$ must belong to the same connected component of the space-time slab, and moreover, the configuration of $A$ at time $\tpluss$ must be a nontrivial function of the initial configuration. Thus, either the histories $\{\sH_u : u\in A\} $ coalesce to a single point $w$ at some time $1\leq T<\tpluss$ --- and then the spin there must depend nontrivially on the initial state, i.e., $X^+_T(w) \neq X^-_T(w)$ --- or the histories for all $u\in A$ all join into one cluster along $(0,\tpluss]$ and at least one of these survives to time 0.
(The same would be true if we did not restrict the coalescence time to be at least 1, yet in this way the conditioning on $\cU$, which only pertains to updates along the interval $(\tpluss-1,\tpluss]$, does not cause any complications.) For the latter, we denote by $\cJ(a,b)$ the event that the histories join in the interval $(a,b)$, and for the former we let
\[\tau' = \min\left\{ t\geq 1 : |\sH_{A}(\tpluss - t)| = 1\right\} \wedge \tpluss\quad~,~\quad T = \tpluss-\tau'\,,\]
and note that the variable $\tau'$ is a stopping time w.r.t.\ the natural filtration associated with exposing the update histories backward from time $\tpluss$; indeed, in contrast to a definition of $\tau_v$ analogous to~\eqref{eq-tauv-def} --- asking for $\{ \sH_u : u \in \cC_v\}$ to coalesce to a single point --- here one only requires this for $\{\sH_u : u\in A\}$ (whereas $\cC_v$ may be affected by the histories along $(0,T]$ as these may admit additional vertices to it).
With this notation, we deduce from the above discussion that
\[
\P(A\in\red \mid \cU) \leq \P\bigg( \bigcup_w \left\{\cJ(T,\tpluss)\,,\, w\in \sH_A(T)\,,\, X^+_T(w)\neq X^-_T(w)\right\}\;\Big|\;\cU\bigg)\,.
\]
(If $T=0$ and $A\in\red$ then $\sH_A(0)\neq\emptyset$, whence $X^+_0(w)\neq X^-_0(w)$ trivially holds for any $w\in\sH_A(0)$.)
By conditioning on $T$ as well as on $\sH_A(T,\tpluss)$, the first two events on the right-hand side become measurable, while the event $X^+_T(w)\neq X^-_T(w)$ only depends on the histories along $(0,T]$ and satisfies
\[ \P\left(X^+_T(w)\neq X^-_T(w) \mid T\,,\, \sH_A(T,\tpluss)\right) = \sm_T(w) \leq e^{\tpluss-T} \sm_{\tpluss}(w) \,,\]
where the final inequality used the fact, mentioned in the proof of Claim~\ref{clm:MagnetDecay}, that $\sm_{t+s}(w) \geq e^{-s} \sm_t(w)$ for any $s,t>0$ and $w$, as the probability of no updates to $w$ along the interval $(t,t+s)$ (maintaining the magnetization without a change) is $e^{-s}$. Now, averaging over this conditional space yields
\begin{align*}
\P(A\in\red \mid \cU) &\leq  \E\bigg[ \sum_w \one_{\{\cJ(T,\tpluss)\}} \one_{\{w\in \sH_A(T)\}} e^{\tpluss - T} \sm_{\tpluss}(w) \;\Big|\;\cU\bigg]\\
&\leq  \E\bigg[ \sum_w \one_{\{A \subset \cC_v\}}  \one_{\{w\in \sH_A(\tpluss-\tau',\tpluss)\}} e^{\tau'} \sm_{\tpluss}(w) \;\Big|\;\cU\bigg]\,,
\end{align*}
where we increased the event $\cJ(T,\tpluss)$ (the joining of $\sH_A$ along $(T,\tpluss]$) into $A \subset \cC_v$ (valid for any $v\in A$) as well as the event $\{w\in\sH_A(T)\}$ into $\{w\in\sH_A(T,\tpluss)\}$,
and finally plugged in that $T=\tpluss - \tau'$.
Since by definition $\tau' \leq \tau_v = \min\left\{ t\geq 1 : |\sH_{\cC_v}(\tpluss - t)| = 1\right\} \wedge \tpluss$ on the event $A\subset \cC_v$, we conclude that
\begin{align}\label{eq-psi-bound-no-hat}
\P(A\in\red \mid \cU) \leq \E\bigg[ \sum_w \one_{\{A \subset \cC_v\}}  \one_{\{w\in \sH_A(\tpluss-\tau_v,\tpluss)\}} e^{\tau_v} \sm_{\tpluss}(w) \;\Big|\;\cU\bigg]\,.
\end{align}
The final step is to eliminate the conditioning on $\cU$ using the modified update history $\hsH$, which we recall does not remove vertices from the history along the unit interval $(\tpluss-1,\tpluss]$ and grants each vertex an automatic update at time $\tpluss$. As such, $\sH_u(t) \subset \hsH_u(t)$ for any vertex $u$ and time $t$.

We claim that each of the terms in the right-hand of~\eqref{eq-psi-bound-no-hat} is increasing in the percolation space-time slab (i.e., they can only increase when adding connections to the update histories). Indeed, this trivially holds for $\{A\subset \cC_v\}$; the variable $\tau_v$ is increasing as it may take only longer for $\cC_v$ to coalesce to a single point; finally, as the interval $(\tpluss-\tau_v,\tpluss]$ does not decrease and neither does $\sH_A$ along it, the event $\{w\in\sH_A(\tpluss-\tau_v,\tpluss)\}$ is also increasing.

Therefore, if we do not remove vertices from the update history along $(\tpluss-1,\tpluss]$ then the right-hand of~\eqref{eq-psi-bound-no-hat} could only increase. Further observe that, as long as no vertices are removed from the history along that unit interval, the connected components of the update history at time $\tpluss-1$ remain exactly the same were we to modify the update times of any vertex there, while keeping them within that unit interval.
In particular, should a vertex at all be updated in that period, we can  move its latest update time to $\tpluss$.

In this version of the update history (retaining all vertices in the given unit interval, and letting the latest most update, if it is in that interval, be performed at time $\tpluss$), the effect of conditioning on $\cU$ in that every $u\in A'$ receives an update at time $\tpluss$.
The fact that $\Po(\lambda \mid \cdot \geq 1) \preccurlyeq \Po(\lambda)+1$ for any $\lambda>0$ (as the ratio $\P(\Po(\lambda)=k)/\P(\Po(\lambda)>k)$ is monotone increasing in $k$) now implies (taking $\lambda\in(0,1)$) that the number of updates that any $u\in A'$ receives along $(\tpluss-1,\tpluss]$ conditioned on $\cU$ as part of $\sH$
is stochastically dominated by the corresponding number of updates as part of $\hsH$.

Altogether we conclude that the right-hand of~\eqref{eq-psi-bound-no-hat} can be increased to yield
\[ \P(A\in\red \mid \cU) \leq \E\bigg[ \sum_w \one_{\{A \subset \hcC_v\}}  \one_{\{w\in \hsH_A(\tpluss-\htau_v,\tpluss)\}} e^{\htau_v} \sm_{\tpluss}(w) \bigg]\,,\]
and combining this with~\eqref{eq-psi-bound-given-U} completes the proof.
\qed

\subsection{Discrete Fourier expansion for the update rules}\label{sec:fourier}
The following lemma, which constructs the modified update rules $\Phi_A$ (as described in~\S\ref{sec:framework}), will play a key role in the proof of Lemma~\ref{lem-todo}.
\begin{lemma}\label{lem-p(k,r)}
 For every $\epsilon>0$ there exists some $\kappa>0$ such that the following holds provided $\beta d < \kappa$.
For any $r\leq d$ there are nonnegative reals $\{ p_{k,r} : k=0,\ldots,r\}$ satisfying
\begin{align}\label{eq-p(k,r)-properties}
p_{0,r}\geq 1-\epsilon\,,\qquad
\sum_k \binom{r}k p_{k,r}=1\,,\qquad\mbox{ and }\qquad
 \binom{r}{k}p_{k,r} \leq D_0 (2\beta r)^k \quad\mbox{ for all $k$}\,,
\end{align}
where $D_0$ is an absolute constant, such that the Glauber dynamics can be coupled to an update function $\Phi$ that selects a subset $A\subset [r]$ of the neighbors of a degree-$r$ vertex with probability $p_{|A|,r}$ and applies to it a symmetric monotone boolean function $\Phi_A$ (i.e., $\Phi_A(-x)=-\Phi_A(x)$ and $\Phi_A(x)$ is increasing in $x$).
\end{lemma}
\begin{proof}
 Setting
 \[ f(x) = \frac12 \left(\tanh(x)+1\right) = \frac{e^x}{e^x + e^{-x}} \]
 we have that the Glauber dynamics update function at a given site with neighbors $\sigma_1,\ldots,\sigma_r$ assigns it a new spin of $1$ with probability $f(\beta \sum_{i=1}^r \sigma_i)$. Writing $f(x)=\sum_{\ell=0}^\infty B_\ell x^\ell$, i.e.,
 \[ B_\ell = [x^\ell] f(x)\,,\]
and so, bearing in mind that $\tanh(z)$ has no singularities in the open disc of radius $\pi/2$ around 0 in $\mathbb{C}$ and thus $\sum B_\ell$ converges absolutely,
\[ B_0=B_1 = 1/2\qquad\mbox{ and } \qquad\sum  |B_\ell| = B\quad\mbox{ for some absolute constant $B>0$}\,.\]

Next, since $\sigma_i\in\{\pm1\}$ the power series is multi-linear in $\sigma_i$, whence
we can write
 \[ \left(\beta \sum \sigma_i\right)^\ell = \sum_{\substack{A\subset [r] \\ |A|\leq \ell}} C_{\ell,A} \prod_{i\in A} \sigma_i = \sum_{k=1}^{\ell\wedge r} C_{\ell,k} \sum_{|A|=k} \prod_{i\in A}\sigma_i\,,\]
where we used that the nonnegative coefficient $C_{\ell,A}$ depends by symmetry on $|A|$ rather than $A$ itself, thus we can write $C_{\ell,k}$ for $|A|=k$. (Note that for $\ell=1$ we have $C_{1,1}=\beta$.)

Now, for any particular $k\leq \ell\wedge r$, we can put $\sigma_1=\ldots=\sigma_r=1$ to find that
\[ \left(\beta r \right)^\ell = \sum_{i=0}^{\ell\wedge r} \sum_{|A|=i} C_{\ell,i} \geq \sum_{|A|=k} C_{\ell,k} = \binom{r}{k} C_{\ell,k}\,,\]
and so
\begin{equation}
  \label{eq-ckl-bound}
  0 \leq C_{\ell,k} \leq \frac{(\beta r)^\ell}{\binom{r}{k}}\,.
\end{equation}
Therefore, letting
\[ C_k =\sum_{\ell=k}^\infty C_{\ell,k} B_\ell\qquad\mbox{ for $k\geq 1$}\]
and recalling that $\sum |B_\ell| = B$, we see that
\begin{equation}
   \label{eq-ck-bound}
   |C_k| \leq  B \sum_{\ell \geq k} \frac{ (\beta r)^{\ell}}{\binom{r}k} \leq 2 B \frac{(\beta r)^{k}}{\binom{r}{k}}\,,
 \end{equation}
with the last inequality valid as long as $\beta r \leq 1/2$.

We now define $p_{k,r}$ as follows:
\begin{equation}\label{eq-p(k,r)-def}
  p_{k,r} = \left\{ \begin{array}{ll}
  2|C_k|(k+1)  & k \geq 2\,,\\
  \noalign{\medskip}
  2 \bigg(C_1 - \sum_{\substack{A'\ni 1 \\ |A'|\geq 2}} |C_{|A'|}|\bigg)  & k=1\,, \\
  \noalign{\medskip}
  1 - \sum_{k\geq 1} \binom{r}k p_{k,r} & k = 0\,.
  \end{array}
  \right.
\end{equation}
Our first step in verifying that this definition satisfies~\eqref{eq-p(k,r)-def} is to show that $0 < p_{1,r} < 1$.
For the upper bound, using~\eqref{eq-ck-bound} we have $p_{1,r}\leq 2|C_1| \leq 4 B \beta < 1$ for $\beta$ small enough.
For the lower bound, observe that since $B_1=1/2$, $C_{1,1}=\beta$ and $C_{\ell,1} \leq (\beta r)^\ell / r$ using~\eqref{eq-ckl-bound},
\begin{align}
C_1 &\geq \frac{\beta}2 - \sum_{\ell=2}^{\infty} C_{\ell,1} | B_\ell |
\geq \frac{\beta}2 - \frac{B}r \sum_{\ell \geq 2}  (\beta r)^\ell \geq
\beta\left(\tfrac12 - 2 \beta r B\right) > \beta/4
\,.\label{eq-C1-lower-bnd}
\end{align}
as long as $\beta < 1/(4rB)$.
On the other hand, again appealing to~\eqref{eq-ck-bound},
\begin{align}
  \sum_{\substack{A'\ni 1 \\ |A'|\geq 2}}\left|C_{|A'|}\right| = \sum_{k=2}^r \binom{r-1}{k-1} |C_k|
  \leq 2B  \sum_{k=2}^r \frac{k}r (\beta r)^k
  = 2 B \beta \sum_{k=2}^{r} k (\beta r)^{k-1} \leq \beta / 8
  \label{eq-sum-A'-upper-bnd}
\end{align}
provided $\beta r $ is sufficiently small.
Combining the last two displays yields $p_{1,r} \geq \beta/8$.

Next, we wish to verify that $\binom{r}k p_{k,r} \leq D_0 (2\beta r)^k$ for some absolute constant $D_0$ and all $k$. Let $D_0 = 4 B$ and note that for $k=0$ the sought inequality is trivial since $D_0 > 1$ (recall $B\geq B_0+B_1 = 1$) whereas $p_{0,r} < 1$ (we have shown that $p_{1,r}>0$ and clearly $p_{k,r}\geq 0$ for all $k\geq 2$).
 For $k=1$ we again recall from~\eqref{eq-ck-bound} that
$ r p_{1,r} \leq 2r |C_1| \leq 4 \beta r B < D_0 (2\beta r)$, and similarly, for $k\geq 2$ we have
\[ \binom{r}k p_{k,r} = 2 \binom{r}k |C_k|(k+1) \leq 4B (k+1)(\beta r)^k \leq 4 B (2\beta r)^k = D_0 (2\beta r)^k\,.\]
For any sufficiently small $\beta r$ this of course also shows that $p_{k,r}\leq 1$ for all $k$, as well as the final fact that $p_{0,d} \geq 1-\epsilon$ since
\begin{align}\label{eq-p-0,d}
 \sum_{k\geq 1} \binom{r}k p_{k,r} \leq D_0 \sum_{k\geq 1} (2\beta r)^k < 4\beta r D_0 < \epsilon
 \end{align}
for a small enough $\beta r$.

Having established that desired properties for $\{ p_{k,r} : 0\leq k \leq r\}$, define the new update function $\Phi$ which will examine a random subset $A$ of the $r$ neighbors of a vertex, selected with probability $p_{|A|,r}$ (giving a proper distribution over the subsets of $[r]$ since $\sum_k\binom{r}{k}p_{k,r}=1$ as shown above), then apply the following function $\Phi_A$ to determine the probability of a plus update given $\sigma_A = \{ \sigma_i : i \in A\}$.
\begin{equation}
  \Phi_A(\sigma_A) = \left\{ \begin{array}{ll} \frac12 & A=\emptyset\,, \\
    \noalign{\medskip}
  \frac12 + \frac12 \sigma_i & A = \{i\}\,, \\
  \noalign{\medskip}
  \frac12 + \frac1{2(|A|+1)}\left[\sum_{i\in A}\sigma_i + \sign(C_{|A|})\prod_{i\in A}\sigma_i\right] & |A|\geq 2\,.\end{array}
  \right.
\end{equation}
In order to establish that $\Phi$ can be coupled to the Glauber dynamics, we need to show that
$f(\beta\sum_{i=1}^r\sigma_i)$ identifies with $\E[\Phi(\sigma_1,\ldots,\sigma_r)] $ over all inputs $\{\sigma_i\}$.
Since $B_0=1/2$, we must show that $\E[\Phi]-1/2$ is equal to $\sum_{\ell=1}^\infty B_\ell (\beta \sum \sigma_i)^\ell$. Indeed,
\begin{align*}
\E[\Phi]-\frac12 &=   \sum_i \sigma_i \bigg(C_1 - \sum_{\substack{A' \ni i \\
|A'| \geq 2}} C_{|A'|}\bigg) + \sum_{|A|\geq 2} |C_{|A|}|\bigg( \sum_{i\in A}\sigma_i + \sign(C_{|A|}) \prod_{i\in A}\sigma_i\bigg)\\
&=\sum_{|A|\geq 1}  C_{|A|} \prod_{i\in A}\sigma_i
= \sum_{\ell=1}^{\infty}\sum_{k=1}^{\ell\wedge r} \sum_{|A|=k}  C_{\ell,k} B_\ell \prod_{i\in A}\sigma_i = \sum_{\ell=1}^{\infty} B_\ell (\beta \sum \sigma_i)^\ell\,,
\end{align*}
with the last two equalities following from the definition of $C_k$ and $C_{\ell,k}$.
This completes the proof.
\end{proof}

\subsection{Exponential decay of cluster sizes: Proof of Lemma~\ref{lem-todo}}

Using the update rule from Lemma~\ref{lem-p(k,r)}, the probability that an update of a vertex $v$ of degree $r\leq d$ will examine precisely $k$ of its neighbors is
\[ \binom{r}k  p_{k,r} \leq D_0(2\beta r)^k \leq D_0 (2\beta d)^k \,,\]
with the inequality thanks to~\eqref{eq-p(k,r)-properties}.
The probability that a given neighbor of $v$, with degree some $r'\leq d$, receives an update in which it examines both $v$ and $k-1$ additional neighbors is at most
\[ \max_{r'\leq d} p_{k,r'} \binom{r'-1}{k-1} \leq
\max_{r'\leq d} \frac{k}{r'} D_0 (2\beta r')^k
= \frac{1}d  D_0 (3\beta d)^k \,,\]
using that $x^{1/x} \leq e^{1/e} < 3/2$ for all $x\geq 2$.
Hence, the  rate at which the history of the vertex $v$ expands to $k$ additional vertices along the time interval $(0,\tpluss)$ is at most
\[ D_0 \left( 1 + r/d\right) (3\beta d)^k \leq 2 D_0 (3\beta d)^k\,.\]
By the same reasoning, the extra update at time $\tpluss$ that is applied to $v$ in $\hsH$ connects it to $k$ of its neighbors ($k=0,\ldots,r$) with probability at most $D_0(2\beta d)^k$, while each of its $r$ neighbors contributes at most $k$ new points with probability at most $D_0 (3\beta d)^k/d$.

We now develop the cluster of the vertex $(w_0,t_0)$ in the space-time slab by exploring the branch at $w_0$, both forwards and backward in time, examining which connections it has to new vertices --- either through its own updates or through those which examine it --- until it terminates via oblivious updates in both directions. We then repeat this process with one of the points discovered in the exploration process (arbitrary chosen), until all such points are exhausted and the cluster is completely revealed.

Let $Y_m$ denote the number of vertices explored in this way after iteration $m$ (i.e., $Y_1$ is the number of vertices discovered via the branch incident to $(w_0,t_0)$, etc.), and let $Z_m$
be the total length of edges in the time dimension (i.e., $(z,a),(z,b)$ for $z\in V$ and $0\leq a < b \leq \tpluss$) explored by then.
We can stochastically dominate these by a process $(\bar{Y}_m,\bar{Z}_m)\succeq (Y_m,Z_m)$ given as follows.

First, for the length variable, we apply Lemma~\ref{lem-p(k,r)} with $\epsilon= (1-\eta)/4$, and put
\begin{align*}
  \bar{Z}_0&=0\,,\\
   \bar{Z}_m &= \bar{Z}_{m-1} + W_m\quad\mbox{ where }\quad W_m\sim 1+\Gamma(2,1-\epsilon)\,,
\end{align*}
 with the gamma variable $\Gamma(2,1-\epsilon)$ measuring the time until the explored branch terminates (in both ends) using the key estimate $p_{0,r}\geq 1-\epsilon$ from Lemma~\ref{lem-p(k,r)},
translated by 1 to account for the unit interval $(\tpluss-1,\tpluss]$ in which vertices are not removed from $\hsH$.

For the vertex count variable, with the above discussion above in mind, observe that conditioned on $W_m$ the number of new vertices exposed along the new branch is dominated by $\sum_{k=1}^d V_m^{(k)}$, in which
\[ V_m^{(k)} \sim k \Po\left( 2 D_0 (3\beta d)^k W_m \right)\qquad (k=1,\ldots,d) \]
are mutually independent, while the extra update at time $\tpluss$ (should the branch extend to that time) introduces at most $\sum_{k=0}^d \hat{V}_m^{(k)}$ additional vertices, where all $V_m^{(k)}$ and $\hat{V}_m^{(k)}$ are independent, given by
\[ \P\left(\hat{V}_{m}^{(0)} = j\right) \leq D_0 (2\beta d)^j\,,\qquad
\P\left(\hat{V}_m^{(k)} =j\right) \leq D_0 (3\beta d)^j/d\qquad (k=1,\ldots,d)\,.
\]
Therefore, with this notation, we write
\begin{align*}
  \bar{Y}_0 &= 1\,,\\
  \bar{Y}_m &= \bar{Y}_{m-1} + U_m\quad\mbox{ where }\quad U_m = \sum_{k=1}^d V_m^{(k)} + \sum_{k=0}^d \hat{V}_m^{(k)}\,.
\end{align*}
Letting $\tau \geq 1$ be the iteration after which the exploration process exhausts all new vertices (so $\tau = 1$ iff both ends of the branch of $(v_0,t_0)$ terminated before introducing any new vertices to the cluster), we wish to show that
\begin{align}
  \label{eq-exp-barZ-barY-bound}
  \E \left[\exp(\eta \bar{Z}_\tau + \lambda \bar{Y}_\tau) \right]\leq \gamma
\end{align}
for $\gamma(\lambda,\eta) < \infty$. We may assume without loss of generality --- recalling that $\epsilon = (1-\eta)/4 \leq \frac14$ --- that
\begin{equation}
  \label{eq-lambda-assumption}
  \lambda \geq 4\log(1/\epsilon) \,,
\end{equation}
as the left-hand of~\eqref{eq-exp-barZ-barY-bound} is monotone increasing in $\lambda$.
Observe that as long as $3\beta d e^{\lambda} < 1/2$ we have
\begin{align*}
  \E \left[ \exp\left(\lambda \hat{V}_m^{(0)}\right) \right] &\leq 1 + D_0 \sum_{k\geq 1}\big(2\beta d e^\lambda \big)^k \leq 1 + 4 D_0 \beta d e^{\lambda}
  \end{align*}
as well as
\begin{align*}
  \prod_{k=1}^d \E \left[ \exp\left(\lambda \hat{V}_m^{(k)}\right) \right] &\leq 1 + 2 D_0 \sum_{k\geq 1}\big(3\beta d e^\lambda \big)^k \leq 1 + 12 D_0 \beta d e^{\lambda} \,,
\end{align*}
and similarly,
\begin{align*}
  \prod_{k=1}^d \E \left[ \exp\left(\lambda V_m^{(k)}\right) \;\Big|\; W_m \right]&= \exp\bigg[  2 W_m D_0   \sum_{k=1}^d (e^{\lambda k}-1)(3\beta d)^k\bigg] \leq
  \exp\left[  12 W_m D_0  \beta d e^\lambda \right]\,.
\end{align*}
We can further assume that
\[ h(\lambda) := D_0 \beta d e^{\lambda} \quad \mbox{ satisfies } \quad 12 h(\lambda) < (1-\eta)/2 = 1 - 2\epsilon - \eta\,,\]
achievable by letting $\beta d$ be sufficiently small. With this notation,
\[ \E\left[ e^{\lambda U_m + \eta W_m } \;\Big|\;  W_m \right] \leq e^{\left(12 h(\lambda) + \eta\right) W_m  + 16 h(\lambda)}\,,\]
and upon taking expectation over $W_m$, having $12 h(\lambda) + \eta < 1-2\epsilon$ implies that the
moment-generating function of the gamma distribution will only contribute a polynomial factor, giving that
\begin{equation}
  \label{eq-exp-Um-Wm}
  \E\left[ e^{\lambda U_m + \eta W_m } \right] \leq e^{28 h(\lambda)+\eta}\bigg(1 - \frac{12 h(\lambda)+\eta}{1-\epsilon}
 \bigg)^{-2} \leq \epsilon^{-2} e^{28 h(\lambda)}\,.
\end{equation}
Combining this with our definition of $\bar{Y}_m = 1 + \sum_{i=1}^m U_i$ and $\bar{Z}_m =\sum_{i=1}^m W_i$, we find that
\begin{align*}
 \E\left[ e^{\lambda \bar{Y}_\tau + \eta \bar{Z}_\tau } \right] &= \E \left [\sum_{m=1}^\infty e^{-\lambda m + 2 \lambda \bar{Y}_m + \eta \bar{Z}_m} \one_{\{\tau = m\}}\right] \leq
 \sum_{m=1}^\infty e^{-\lambda m} \E \left[ e^{2 \lambda \bar{Y}_m + \eta \bar{Z}_m} \right] \\
& =\sum_{m=1}^\infty e^{\lambda (2-m)} \left( \E \left[e^{2 \lambda U_1 + \eta W_1}\right]\right)^m
\,,
\end{align*}
which, recalling~\eqref{eq-exp-Um-Wm} and plugging in the expression for $h(2\lambda)$, is at most
\begin{align*}
&e^{2\lambda}\sum_{m=1}^\infty \exp\left[ m \left(-\lambda + 2\log\big(\tfrac{4}{1-\eta}) + 28 h(2\lambda) \right)\right] \leq e^{2\lambda}\sum_{m=1}^\infty  \exp\left[ m \left(-\lambda/2 + 28h(2\lambda) \right)\right]
= \gamma < \infty
\end{align*}
using~\eqref{eq-lambda-assumption} for the first inequality and, say, that $28 h(2\lambda) \leq \lambda/3$ (achieved by taking $\beta d$ small enough) for the second one.
(Note that $\gamma=\gamma(\lambda,\eta)$, as the assumption~\eqref{eq-lambda-assumption} introduces a dependence on $\eta$).
This establishes~\eqref{eq-exp-barZ-barY-bound} and thereby concludes the proof.
\qed

\section{The effect of initial conditions on mixing}\label{sec:ann-que}

In this section we consider random initial conditions (both quenched and annealed), and prove Theorem~\ref{mainthm-ann-que}.
The first observation is that, thanks to Theorem~\ref{mainthm-gen}, the worst-case mixing time satisfies
\[ \tmix(\alpha) = \tcut + O(1)\quad\mbox{ for any fixed $0<\alpha<1$}\,,\]
with $\tcut$ as defined in~\eqref{eq-t*-def}, and moreover, the same holds for $\tmix^{(+)}(\alpha)$, the mixing time started from all-plus.
  By Claim~\ref{clm:MagnetDecay} we have $\frac12\log n \leq \tcut \leq (\frac12 + \epsilon_\beta)\log n$ with $\epsilon_\beta = \beta d/(2-2\beta d)$ vanishing as $\beta\downarrow0$.
Thus, we may prove the bounds on the annealed / quenched mixing times when replacing $\tcut$ by $\frac12\log n$.

\subsection{Annealed analysis}

As mentioned in the introduction, rather than comparing two worst case boundary conditions we will compare a random one directly with the stationary distribution: By considering updates in the range $t\in(-\infty,\tcut]$ we can use the coupling from the past construction to generate a coupling with the stationary distribution. Let $X_t$ denote the process started from uniform initial conditions at time 0 and let $Y_t$ be the process generated by coupling from the past.

The information percolation clusters of $V$ will now be defined as the connected components of the graph on the vertex set $V$ where $(u,v)$ is an edge iff $\sH_u(t)\cap \sH_v(t) \neq \emptyset $ for some $-\infty < t \leq \tcut$ (in contrast to the previous definition where we had $0<t \leq \tcut$). The notion of being a red cluster is redefined to be any $\cC_v$ such that
$
\left| \bigcup_{u\in \cC_v} \sH_u(t') \right | \geq 2
$
for all $0\leq t' \leq \tcut$.  Blue clusters will be defined as before and green clusters will again be the remaining clusters.  We claim that we can couple the spins at time $\tcut$ of all non-red clusters. Indeed if a cluster $\cC_v$ is not red then there is some time $t'>0$ such that $|\bigcup_{u\in \cC_v} \sH_u(t')|=1$.  Call this vertex $w$.  By symmetry both $X_{t'}(w)$ and $Y_{t'}(w)$ are equally likely to be plus or minus and so we may couple them to be equal independently of the spins of the other clusters.  We may then also couple the spins in that cluster to be the same in both $X_t$ and $Y_t$ to be equal for all $t>t'$.  Thus the configurations will agree outside of the red clusters.

Let $\anim(A)$ denote the size of the smallest connected set of vertices (animal) containing $A$.  In a graph of maximum degree $d$, the number of trees of size $k$ containing the vertex $v$ is bounded above by $(ed)^k$ and hence the number of animals $A$ containing a specified vertex with $\anim(A)=k$ is at most $(ed)^k$.

\begin{lemma}\label{l:annealedPsi}
For any $d,C,\epsilon > 0$ there exists $\beta_0>0$  such that the following holds for large enough $n$. If $0<\beta<\beta_0$ and $\tpluss=(\frac14+\epsilon)\log n$ then for any $A$,
\[
\sup_{\sH_A^-} \P\left(A\in \red \mid \sH_A^-\,,\{A\in\red\}\cup\{A\subset V_\blue\}\right) \leq   \frac1{\sqrt{n}\log n} e^{- C \anim(A)}\,.
\]
\end{lemma}

\begin{proof}
Similarly to  the proof of Lemma~\ref{lem-Psi} we have the analogue of equation~\eqref{eq-psi-bound-given-U}
\[
\P\left(A\in\red\mid\sH_A^-\,,\{A\in\red\}\cup\{A\subset V_\blue\}\right) \leq 2^{|A|} \P(A\in\red \mid \cU)\, ,
\]
where $\cU$ is defined as in the proof of Lemma~\ref{lem-Psi}.  Then for any $v\in A$,
\begin{align*}
\P(A \in \red \mid \cU) \leq  \P\left(|\sH_{\cC_v}| \geq \anim(A)\,,\,\len(\cX_{v,0})\geq 2 \tpluss \;\Big|\; \cU\right)\,,
\end{align*}
since the total length of a red cluster must be at least $2 \tpluss$ and it must contain at least $\anim(A)$ vertices.  Both $|\sH_{\cC_v}|$ and $\len(\cX_{v,0})$ are increasing in the component sizes, and so, by the same monotonicity argument as Lemma~\ref{lem-Psi}, we have that
\begin{align*}
\P\left(A\in\red\mid\sH_A^-\,,\{A\in\red\}\cup\{A\subset V_\blue\}\right) \leq 2^{|A|}  \P\left(|\hsH_{\hcC_v}| \geq \anim(A)\,,\,\len(\hcX_{v,0})\geq 2 \tpluss \right)\,.
\end{align*}
Taking $\lambda =\log 2 +C$ and $\frac14 (\frac14 + \epsilon)^{-1}< \eta < 1$ in Lemma~\ref{lem-todo} then shows that, for $\beta_0$ small enough,
\begin{align*}
\P\left(|\hsH_{\hcC_v}| \geq \anim(A)\,,\,\len(\hcX_{v,0})\geq 2 \tpluss \right) &\leq \frac{\E\left[ \exp\left( \eta \len(\hcX_{v,0}) + \lambda |\hsH_{\hcC_v}|\right)\right]}{\exp\left( 2 \eta \tpluss + \lambda \anim(A)\right)} \\
&\leq \gamma \exp\Big( - 2 \eta (\tfrac14 + \epsilon) \log n - (\log 2 +C) \anim(A)\Big)\\& \leq \frac1{\sqrt{n}\log n} 2^{- \anim(A)} e^{- C \anim(A)}
\end{align*}
(with room, as we could have replaced the $\sqrt{n}\log n$ by some $n^{1/2+\epsilon'}$), which completes the proof.
\end{proof}

We now establish an upper bound on $\tmix^\uni$, the mixing time starting from the uniform distribution.

\begin{proposition}\label{prop:anneal}
For any $d,\epsilon > 0$ there exists $\beta_0>0$ such that the following holds. If $0<\beta < \beta_0$ and $\tpluss=(\frac14+\epsilon)\log n$ then
$
\|\P(X_{\tpluss} \in \cdot) - \pi \|_{\tv} \to 0
$
as $n\to\infty$.
\end{proposition}

\begin{proof}
Having coupled $X_t$ and $Y_t$ as described above we have that
\begin{align*}
\left\| \P(X_{\tpluss}\in\cdot)-\P(Y_{\tpluss}\in\cdot)\right\|_\tv &\leq \E\left[\left\|\P\left(X_{\tpluss}(V \setminus V_\green) \in \cdot \mid \sH_\green\right) - \nu_{V\setminus V_\green}\right\|_\tv \right]\\
 &\qquad +\E\left[\left\|\P\left(Y_{\tpluss}(V \setminus V_\green) \in \cdot \mid \sH_\green\right) - \nu_{V\setminus V_\green}\right\|_\tv \right]
\end{align*}
where $\nu_A$ is the uniform measure on the configurations on $A$.  Similarly to the argument used to derive equation~\eqref{eq-exp-moment-2}, we find that
\begin{align*}
\left\| \P(X_{\tpluss}\in\cdot)-\P(Y_{\tpluss}\in\cdot)\right\|_\tv \leq
\left(\sup_{\sH_\green} \E
\left[2^{\left|V_\red \cap V_{\red'}\right|} \;\big|\; \sH_\green\right] - 1 \right)^{1/2}\,.
\end{align*}
With the same coupling as in the proof of Theorem~\ref{mainthm-gen}, analogously to equation~\eqref{eq-req-exp-bound}, we have
\begin{equation}
\| \P(X_{\tpluss}\in\cdot)-\P(Y_{\tpluss}\in\cdot)\|_\tv^2 \leq 2 \sum_v \bigg(\sum_{A \ni v} 2^{|A|} \Psi_{A}\bigg)^2\,.
 \end{equation}
Applying Lemma~\ref{l:annealedPsi} with $C = \lceil \log (4ed)\rceil$ (while recalling that $\#\{A \ni v: \anim(A)=k\}\leq (ed)^k$), we get
\begin{align*}
\sum_{A \ni v} 2^{|A|} \Psi_{A} &\leq  \sum_{k} \sum_{\substack{A:\anim(A)=k\\ v\in A}} \frac{2^k e^{- C k}} {\sqrt{n}\log n} \leq  \sum_{k} \frac{(2ed)^k e^{- C k}} {\sqrt{n}\log n} \leq \frac1{\sqrt{n}\log n}   \,,
\end{align*}
provided that $\beta>\beta_0$ with $\beta_0$ from that lemma.
It follows that
\begin{align*}
\| \P(X_{\tpluss}\in\cdot)-\P(Y_{\tpluss}\in\cdot)\|_\tv &\leq O\left(\log^{-2} n\right)\,,
\end{align*}
and in particular $
\| \P(X_{\tpluss}\in\cdot)-\pi\|_\tv = o(1)
$, as required.
\end{proof}
\begin{remark}
In the above proof one could instead carry the analysis as in the proof of Theorem~\ref{mainthm-gen} (partitioning the event in Lemma~\ref{lem-Psi} according to the events $\{|\hcC_v|=k\}$
when estimating the sum over $v\ni A$), that way replacing the factor of $(e d)^k$ lattice animals by $2^k$ subsets of $\hcC_v$. Consequently, the statement of Proposition~\ref{prop:anneal} remains valid for any $\beta < c_0 /d$ where $c_0$ depends on $\epsilon$ but not on $d$.
\end{remark}
\subsection{Quenched analysis}

Here we show that the mixing time from a typical random initial state is at most a factor of $1+\epsilon_\beta$ faster than that from the worst starting state.
As before, let $X_t$ be started from a uniformly chosen initial state $X_0$ and let $Y_t$ be started from the stationary distribution $\pi$.

\begin{proposition}
Let $\tminuss= \frac1{2}\log n - w_n$ for some $w_n \uparrow\infty$. Then
$
\|\P_{X_0}(X_{\tminuss} \in \cdot) - \pi \|_{\tv} \stackrel{p}{\to} 1
$
as $n\to\infty$.
\end{proposition}
\begin{proof}
Note that by the monotonicity of the update rules, for any update history of $u$, the spin at $u$ is a monotone function of $X_0$.  With probability $e^{-t}$ the vertex $u$ is never updated in which case $X_t(u)=X_0(u)$.  Since by symmetry $\E[X_t(u)]=0$, it follows that
\[
\E[X_t(u) \mid X_0(u)=+1] \geq e^{-t}\,, \quad \E[X_t(u) \mid X_0(u)=-1] \leq  - e^{-t}\,.
\]
Thus we have that $\E [X_0(u) X_t(u)] \geq e^{-t}$ and so
\[
\E\Big[ \sum_u X_0(u) X_{\tminuss}(u)\Big] \geq n e^{-\tminuss} = \sqrt{n}e^{w_n}\,.
\]
Let $\cE_{u,v}$ be the event that $u\in \cC_v$ or  $v\in\sH_u(0)$ or $u\in  \sH_v(0)$ for the history developed from time $\tminuss$. Similarly to Claim~\ref{clm:Correlation}, let $X_t'$ and $X_t''$ be two independent copies of the dynamics.  By exploring the histories we may couple $X_t$ with $X_t'$ and $X_t''$ so that, on the event $\cE^c_{u,v}$, the history of $v$ in $X_t$ is equal to the history of $v$ in $X_t'$ and the history of $u$ in $X_t$ is equal to the history of $u$ in $X_t''$.  Hence,
\begin{align*}
\E\Big[X_0(u) X_{\tminuss}(u) \, X_0(v) X_{\tminuss}(v)\Big] &\leq \E\Big[ X_0'(u) X_{\tminuss}'(u) \Big]\E\Big[ X_0''(v) X_{\tminuss}''(v)\Big] + 2\P(\cE_{u,v})\,,
\end{align*}
yielding $\Cov\big(X_0(u) X_{\tminuss}(u)\,,\, X_0(v) X_{\tminuss}(v)\big) \leq 2\P(\cE_{u,v})$. By Lemma~\ref{lem-todo},
\[
\sum_u \Cov\Big(X_0(u) X_{\tminuss}(u)\,,\,X_0(v) X_{\tminuss}(v)\Big) \leq c_1\,.
\]
and so
\[
\var\Big( \sum_u X_0(u) X_{\tminuss}(u) \Big) \leq c_1 n\,.
\]
Thus, by Chebyshev's inequality we infer that
\[
\P\Big(\sum_u X_0(u) X_{\tminuss}(u) > \tfrac12 \sqrt{n} e^{w_n}\Big) \geq  1 - O\left(e^{-2w_n}\right)\,,
\]
and so by Markov's inequality,
\begin{equation}\label{e:quenchedTestA}
\P\bigg(\P\Big(\sum_u X_0(u)X_{\tminuss}(u) > \tfrac12 \sqrt{n} e^{w_n} \;\Big|\; X_0\Big) \geq 1 - e^{-w_n}\bigg) \geq 1 - O\left(e^{-w_n}\right) \to 1\,.
\end{equation}
By the exponential decay of correlations of $Y$ and the fact that it is independent of $X$ we have that
\[
\var \Big(\sum_{u} X_0(u) Y_{\tminuss}(u) \;\Big|\; X_0\Big) \leq c_2 n
\]
for some $c_2>0$. Thus, since $\E \left[\sum_{u} X_0(u) Y_{\tminuss}(u) \mid X_0\right]=0$,
it follows that
\begin{equation}\label{e:quenchedTestB}
\P\Big( \sum_{u} X_0(u) Y_{\tminuss}(u)  > \tfrac12 \sqrt{n} e^{w_n} \;\Big|\; X_0 \Big) = O\left(e^{-2w_n}\right)\to 0
\end{equation}
uniformly in $X_0$.  Comparing equations \eqref{e:quenchedTestA} and \eqref{e:quenchedTestB} completes the result.
\end{proof}

\end{document}